\documentclass[reqno]{amsart}
\usepackage[a4paper, margin=1in]{geometry}

\usepackage{hyperref}
\hypersetup{ colorlinks=true, linkcolor=blue, filecolor=magenta, urlcolor=cyan, }
\usepackage{graphics} 
\usepackage{graphicx} 
\usepackage{epsfig} 
\usepackage{epstopdf}
\usepackage{multirow}
\usepackage{booktabs}
\usepackage{ifpdf}
\usepackage{cite} 
\usepackage{amssymb} 
\usepackage{amsthm} 
\usepackage{amsfonts}
\usepackage{amsmath} 

\usepackage{matlab-prettifier}
\usepackage{algorithm}
\usepackage{algpseudocode}

\usepackage{todonotes}

\usepackage{tikz}
\usetikzlibrary{decorations.pathreplacing,
	calligraphy,
	matrix}
\usepackage{lscape}
\usepackage{tikz-cd}
\usetikzlibrary{positioning}
\usepackage{url}
\usepackage{float}
\restylefloat{table}
\usepackage{subcaption}
\usepackage{mwe}
\usepackage{enumitem}
\usepackage{listofitems} 
\usetikzlibrary{arrows.meta} 
\usepackage[outline]{contour} 
\contourlength{1.4pt}

\usepackage{longtable}

\usepackage{xcolor}
\colorlet{myred}{red!80!black}
\colorlet{myblue}{blue!80!black}
\colorlet{mygreen}{green!60!black}
\colorlet{myorange}{orange!70!red!60!black}
\colorlet{mydarkred}{red!30!black}
\colorlet{mydarkblue}{blue!40!black}
\colorlet{mydarkgreen}{green!30!black}

\usepackage{siunitx}

\tikzset{
  >=latex, 
  node/.style={thick,circle,draw=myblue,minimum size=22,inner sep=0.5,outer sep=0.6},
  node in/.style={node,green!20!black,draw=mygreen!30!black,fill=mygreen!25},
  node hidden/.style={node,blue!20!black,draw=myblue!30!black,fill=myblue!20},
  node convol/.style={node,orange!20!black,draw=myorange!30!black,fill=myorange!20},
  node out/.style={node,red!20!black,draw=myred!30!black,fill=myred!20},
  connect/.style={thick,mydarkblue}, 
  connect arrow/.style={-{Latex[length=4,width=3.5]},thick,mydarkblue,shorten <=0.5,shorten >=1},
  node 1/.style={node in}, 
  node 2/.style={node hidden},
  node 3/.style={node out}
}

\DeclareUnicodeCharacter{0229}{\c{e}}

\theoremstyle{plain}
\newtheorem{theorem}{Theorem}
\newtheorem{lemma}{Lemma}
\newtheorem*{lemma*}{Lemma}
\newtheorem{proposition}{Proposition}

\theoremstyle{definition}     
\newtheorem{definition}{Definition}
\newtheorem{example}{Example}
\newtheorem{remark}{Remark}
\newtheorem{notation}{Notation}

\numberwithin{theorem}{section}
\numberwithin{definition}{section}
\numberwithin{lemma}{section}
\numberwithin{proposition}{section}
\numberwithin{corollary}{section}
\numberwithin{notation}{section}
\numberwithin{remark}{section}
\numberwithin{example}{section}

\begin{document}

\bibliographystyle{plain}

\title[On the Computation of Tensor Functions]{On the Computation of Tensor Functions under Tensor-Tensor Multiplications with Linear Maps}

\author{Jeong-Hoon Ju}
\address{
\parbox{\linewidth}{Department of Mathematics, Pusan National University, 2 Busandaehak-ro 63beon-gil, Geumjeung-gu, 46241 Busan, Republic of Korea;\\
Humanoid Olfactory Display Center, Pusan National University, 49, Busandaehak-ro, Mulgeum-eup, 50612 Yangsan-si, Gyeongsangnam-do, Republic of Korea}}
\email{jjh793012@naver.com}

\author{Susana L\'opez-Moreno*}
\address{
\parbox{\linewidth}{Department of Mathematics, Pusan National University, 2 Busandaehak-ro 63beon-gil, Geumjeung-gu, 46241 Busan, Republic of Korea;\\
Humanoid Olfactory Display Center, Pusan National University, 49, Busandaehak-ro, Mulgeum-eup, 50612 Yangsan-si, Gyeongsangnam-do, Republic of Korea;\\
Industrial Mathematics Center, Pusan National University, 2 Busandaehak-ro 63beon-gil, Geumjeung-gu, 46241 Busan, Republic of Korea}}
\email{susanalopezmoreno@pusan.ac.kr}

\thanks{*Corresponding author}

\date{\today}

\begin{abstract}
In this paper we study the computation of both algebraic and non-algebraic tensor functions under the tensor-tensor multiplication with linear maps. In the case of algebraic tensor functions, we prove that the asymptotic exponent of both the tensor-tensor multiplication and the tensor polynomial evaluation problem under this multiplication is the same as that of the matrix multiplication, unless the linear map is injective. As for non-algebraic functions, we define the tensor geometric mean and the tensor Wasserstein mean for pseudo-positive-definite tensors under the tensor-tensor multiplication with invertible linear maps, and we show that the tensor geometric mean can be calculated by solving a specific Riccati tensor equation. Furthermore, we show that the tensor geometric mean does not satisfy the resultantal (determinantal) identity in general, which the matrix geometric mean always satisfies. Then we define a pseudo-SVD for the injective linear map case and we apply it on image data compression. 
\end{abstract}



\maketitle

\section{Introduction}\label{Sec:intr}

A tensor is a multilinear map and, equivalently, a multidimensional array once a basis is fixed for each vector space on which the multilinear map is defined. In the case of linear maps, the corresponding array is a matrix. In the case of bilinear maps, the corresponding array is a cube. If the dimension of the corresponding multidimensional array is $d$, then we say that it is a $d$-th order tensor. For example, a bilinear map is a third-order tensor. 

Even though a tensor is a natural generalization of a matrix, there is a significant difference between them. Unlike the matrix (linear map) case, it is hard to think of a natural composition of two higher-order tensors (multilinear maps). However, as matrices play a role of representing not only linear maps but also data, higher-order tensors can also be used to represent data, colored image data and videos being some of the most important examples. In the case of colored image data, note that each pixel has three values, the RGB channels, and thus, they can be naturally represented as a third-order tensor. Hence, in order to analyze tensor data effectively, it has been necessary to develop the notion of tensor-tensor multiplication regardless of their multilinear structure. 

The notion of tensor-tensor multiplication we mainly research in this paper is the tensor-tensor multiplication with linear maps, which is based on the one with invertible linear maps in \cite{MR3394164}. It is a generalization of the known T-product in \cite{MR2794595}. Briefly, both tensor-tensor multiplications in \cite{MR2794595, MR3394164} work by viewing tensors as stacks of matrices, taking linear combinations of the matrices, and then performing matrix multiplications. Thus, these tensor-tensor multiplications make it possible to do tensor analysis by using matrix theory. 

Since we are using a product that views tensors as stacks of matrices, it is natural to wonder how its complexity compares with that of the matrix multiplication. However, to the best of our knowledge, this research is the first attempt to determine the computational complexity of the tensor-tensor multiplication with linear maps. As we will check in Section \ref{Sec:bilinear_structure}, this tensor-tensor multiplication is a bilinear map and, thus, we can think of its tensor rank and asymptotic exponent. Tensor rank is a measure of algebraic complexity that quantifies how complicated a multilinear map (tensor) is, and it can be considered a generalization of matrix rank. On the other hand, the asymptotic exponent characterizes the complexity of the operation on sufficiently large-sized objects. For detailed definitions of tensor rank and asymptotic exponent, see Section \ref{Sec:Tensor rank}. 

Matrix multiplication is a fundamental object in linear algebra. In addition, since it is also a bilinear map, we can think of its tensor rank and asymptotic exponent. In 1969, Strassen found a new algorithm for the $2 \times 2$ matrix multiplication that only uses $7$ multiplications \cite{MR248973}, whereas the classical algorithm uses $8$ multiplications. This means that the tensor rank of the $2 \times 2$ matrix multiplication is less than or equal to $7$ (cf. it is exactly $7$ \cite{MR274293, MR297115}). Moreover, from recursive $2 \times 2$ block matrix multiplications, we can note that Strassen's algorithm gives us a way to compute the $n \times n$ matrix multiplication for sufficiently large $n$, which is much faster than the classical one. This is how the notion of the asymptotic exponent was derived. The problem of figuring out the asymptotic exponent of matrix multiplication is still open, and it is now one of the flagship problems in algebraic complexity theory. We refer to \cite{MR1440179, blaeser13} for the history and key ideas on the research of the value. 

A fruitful research topic related to the asymptotic exponent of the matrix multiplication is the search for bilinear maps or computation sequences (see Section \ref{Sec:Tensor rank}) that have the same asymptotic exponent as that of matrix multiplication. For example, LUP decomposition, matrix inversion, Jordan matrix multiplication, and some Lie algebra bracket operations satisfy this property \cite{MR4144013}. In this paper, one of our main results is that the asymptotic exponent of the tensor-tensor multiplication with invertible or surjective linear maps is the same as that of the matrix multiplication.

One of the advantages of the tensor-tensor multiplications is that this multiplication makes it possible to develop a theory of tensor means, such as the tensor geometric mean and the tensor Wasserstein mean. The geometric mean for (symmetric or Hermitian) positive-definite matrices was defined in \cite{MR420302}, and the one for positive operators was defined in \cite{ando1978topics}. It has also been proved that the geometric mean satisfies the algebraic and geometric properties which ``means" must satisfy: it satisfies the idempotence, commutativity, congruence invariance, and inversion order-reversing properties \cite{MR1864051}, and it is the midpoint of the geodesic between two points in a specific Riemannian manifold of positive-definite matrices \cite{MR2198952, MR2137480}. In \cite{MR3992484}, a new Riemannian metric was introduced on the manifold of positive-definite matrices, and the Wasserstein mean was defined as the midpoint of the geodesic between two points in this manifold. In Section \ref{Sec:means} of this paper, we define pseudo-positive-definiteness for third-order tensors, and then develop a theory of tensor geometric and Wasserstein means for these tensors according to the tensor-tensor multiplication. It is a generalization of the tensor geometric mean theory under the T-product in \cite{ju2024}.

Lastly, the tensor-tensor multiplication may look simple but it gives us useful tools in practice. The authors of \cite{MR3394164} introduced the singular value decomposition (SVD) of third-order tensors with respect to the tensor-tensor multiplication defined in the same paper, and successfully used that SVD for the compression of some specific image data. It was also proved in \cite{MR4304063} that the $*_M$-SVD satisfies an Eckart-Young-like theorem. In addition, it can be proved that their tensor-tensor multiplication represents the composition of linear maps between specific finitely generated free modules (see Proposition \ref{Prop:InvertibleCase}), which was proved for the case of the T-product \cite{MR2680253}. In summary, this tensor-tensor multiplication is meaningful not only algebraically but also for data science. We can also observe a similar research in \cite{keegan24} for the tensor-tensor multiplication with surjective orthogonal maps. In this paper, we will define a pseudo-SVD for the tensor-tensor multiplication with injective maps, and apply it to specific image data for its compression. In particular, we will use an injective Johnson-Lindenstrauss-like embedding in the application.

This article is organized as follows: Section \ref{Sec:Preliminaries} introduces some preliminary definitions on tensor-tensor multiplications with linear maps, compares the multiplications for invertible/surjective/injective linear maps, and reviews any concepts in algebraic complexity theory that we may need. In Section \ref{Sec:algebraic}, we study the computational complexity of algebraic tensor functions, in particular, we prove some results on the tensor-tensor multiplication with linear maps and on tensor polynomials. In Section \ref{Sec:non-algebraic}, we study the theory, computation and application of some non-algebraic tensor functions: Section \ref{Sec:means} introduces  the geometric mean and the Wasserstein mean for pseudo-positive-definite tensors in the invertible linear map case, and Section \ref{Sec:applications_svd} includes results of experiments in data compression using the pseudo-SVD under the tensor-tensor multiplication with injective linear maps.

All experiments were implemented in MATLAB R2024b~\cite{matlab2024b} on a MacBook Air with M1 chip and 8 GB RAM. The experiments required the use of the Hyperspectral Imaging Library for Image Processing Toolbox from MATLAB \cite{hyperspectral}. For reproducibility, we set a fixed randomness seed. We used and modified the existing code by the authors in \cite{keegan24}, which is available at their GitHub repository in \cite{newman2025projectedproducts}. Any additional code and all results are available at the GitHub repository: 
\url{https://github.com/SusanaLM/Injective-tensor-multiplication}.

We summarize the obtained theoretical results in Table \ref{tab:results_per_matrix}, organized according to each type of linear map.

\begin{table}[H]
  \centering
  \caption{Summary of theoretical results obtained per type of linear map. Results filled with only ``---" were cases that we did not consider due to previous failing properties necessary for their research.}
  \begin{tabular}{lccc}
    \toprule
    Properties
      & Invertible
      & Surjective
      & Injective \\
    \hline
    \addlinespace

    Type of inverse & Two-sided inverse     & Right inverse    & Left inverse \\
     
    \addlinespace
    \hline
    \addlinespace

    Associativity & Y \cite{MR3394164}     & Y \cite{keegan24}    & N \\ & & (Proposition \ref{prop:SurjectiveCase}) & (Example \ref{ex:InjectiveNotAssociative}) \\

    \addlinespace
    \hline
    \addlinespace

    $(\mathbb{K}^{1 \times 1 \times p},*_M)$ forms a & Y\cite{MR3394164} & Y\cite{keegan24} & N \\
    commutative ring & & & (Example \ref{ex:InjectiveNotAssociative}) \\

    \addlinespace
    \hline
    \addlinespace

     Represents composition & Y\cite{MR3394164}     & Y    & N \\ 
     of two linear maps &   & (Proposition \ref{prop:SurjectiveCase})  & (Example \ref{ex:InjectiveNotAssociative}) \\ between modules &      &     & \\

    \addlinespace
    \hline
    \addlinespace

    Guaranteed existence of  & Y \cite{MR3394164}     & Y \cite{keegan24}    & N\\
    $*_M$-identity tensor & & & (Remark \ref{rmk:RemarkIdentityTensor}) \\

    \addlinespace
    \hline
    \addlinespace

    Uniqueness of $*_M$-identity & Y\cite{MR3394164}     & N \cite{keegan24}    & Y \\ tensor &      &     & (Lemma \ref{lemma:inj_identity})\\
    (if it exists) &      &     &\\
    \addlinespace
    \hline
    \addlinespace

    $*_M$-identity tensor  & Y \cite{MR3394164}     & N \cite{keegan24}   & Y \\ 
     acts as identity & &  & (Remark \ref{rmk:UniquenessOfInverse}) \\
     (if it exists) & & & \\
    \addlinespace
    \hline
    \addlinespace

     Uniqueness of $*_M$-inverse & Y \cite{MR3394164}    & N \cite{keegan24}   & Y \\ (if it exists) &  &  & (Remark \ref{rmk:UniquenessOfInverse})\\
    \addlinespace
    \hline
    \addlinespace

    Bilinearity & Y     & Y    & Y\\
    & (Lemma \ref{lemma:bilinearity}) & (Lemma \ref{lemma:bilinearity}) & (Lemma \ref{lemma:bilinearity})\\
    \addlinespace
    \hline
    \addlinespace

    Asymptotic exponent of & $=\omega(\operatorname{M}_{\langle n \rangle})$     & $=\omega(\operatorname{M}_{\langle n \rangle})$    & $\leq\omega(\operatorname{M}_{\langle n \rangle})$ \\ $*_M$-product on the & (Theorem \ref{thm:AsymptoticExponentRank}) & (Theorem \ref{thm:AsymptoticExponentRank}) & (Example \ref{ex:ZeroTMTensor}) \\ tensor rank & & & (Remark \ref{rmk:InjAsymptoticEquality}) \\
    
    \addlinespace
    \hline
    \addlinespace

    Asymptotic exponent of & $=\omega(\operatorname{M}_{\langle n \rangle})$     & $=\omega(\operatorname{M}_{\langle n \rangle})$    & --- \\ $*_M$-product on the & (Theorem \ref{thm:AsymptoticExponentTotal}) & (Theorem \ref{thm:AsymptoticExponentTotal}) & \\ total complexity & & & \\

    \addlinespace
    \hline
    \addlinespace

    Asymptotic exponent of & $=\omega(\operatorname{M}_{\langle n \rangle})$     & $=\omega(\operatorname{M}_{\langle n \rangle})$    & --- \\ 
    tensor polynomial & (Theorem \ref{Thm:AsymptoticPolynomial}) & (Theorem \ref{Thm:AsymptoticPolynomial}) & (Example \ref{ex:InjectiveNotAssociative})\\

    \addlinespace
    \hline
    \addlinespace
    Existence of $*_M$-inverse of & Y & Y & N \\
    $*_M$-pseudo-positive-definite & (Lemma \ref{lemma:invertible})& (Lemma \ref{lemma:invertible}) & (Example \ref{ex:InjectivePositiveDefiniteNonInvertible}) \\ tensors & & & \\

    \addlinespace
    \hline
    \addlinespace
     Well-definedness of $*_M$-square & Y & Y & N \\ root for $*_M$-pseudo- & (Lemma \ref{lemma:sqrt}) & (Lemma \ref{lemma:sqrt}) & (Example \ref{ex:InjectivePositiveDefiniteNonInvertible})\\ positive-definite tensors & & & \\

    \addlinespace
    \hline
    \addlinespace
    Uniqueness of $*_M$-square root & Y & N & ---\\ for  $*_M$-pseudo- & (Remark \ref{rmk:AssumptionForUniqueSquareRoot}) & (Example \ref{ex:SurjectiveSquareRootNotUnique}) & (Example \ref{ex:InjectivePositiveDefiniteNonInvertible}) \\ positive-definite tensors & & & \\

    \addlinespace
    \hline
    \addlinespace
    Well-definedness of & Y & N & N \\ geometric mean of & (Remark \ref{rmk:WellDefinedGM}) & (Example \ref{ex:SurjectiveSquareRootNotUnique}) & (Example \ref{ex:InjectiveNoPD}) \\ $*_M$-pseudo-positive-definite & & & (Example \ref{ex:InjectivePositiveDefiniteNonInvertible}) \\ tensors & & & \\

    \bottomrule
  \end{tabular}
\end{table}

\begin{table}[h]\ContinuedFloat
  \centering
  \caption{Continued table.}
  \label{tab:results_per_matrix}
  \begin{tabular}{lccc}
    \toprule
    Properties
      & Invertible
      & Surjective
      & Injective \\
    \hline
    \addlinespace
    
     Tensor equation & Y &--- & ---\\
     approach & (Theorem \ref{thm:RiccatiTensorEquation}) & & \\

    \addlinespace
    \hline
    \addlinespace
     Riemannian geometric & Y &--- & ---\\ approach & (Theorem \ref{thm:MidpointGeodesic})& & \\

    \addlinespace
    \hline
    \addlinespace
    Resultantal identity & N &--- & ---\\
    & (Example \ref{ex:CounterexampleResultant}) & & \\

    \addlinespace
    \hline
    \addlinespace
    Cayley's hyperdeterminantal & N &--- & ---\\ identity & (Example \ref{ex:hyper})& & \\
    
    \addlinespace
    \hline
    \addlinespace
    Well-definedness of & Y & N & N \\ Wasserstein mean of $*_M$-pseudo- & (Remark \ref{rmk:AssumptionForUniqueSquareRoot}) & (Example \ref{ex:SurjectiveSquareRootNotUnique}) & (Example \ref{ex:InjectiveNoPD}) \\
    positive-definite tensors & & & \\

    \addlinespace
    \hline
    \addlinespace
    $*_M$-SVD & Y\cite{MR3394164} & Y\cite{keegan24} & N \\ 
    & & & (Example \ref{ex:InjectiveNotAssociative}) \\

     \addlinespace
    \hline
    \addlinespace
    $*_M$-pseudo-SVD & Y\cite{MR3394164} & Y\cite{keegan24} & Y \\ 
    & & & (Definition \ref{defn:pseudoSVD})\\
    
    \bottomrule
  \end{tabular}
\end{table}

\section*{Acknowledgments}

This work was supported by the Korea National Research Foundation (NRF) grant funded by the Korean government (MSIT) (RS-2024-00406152). J.-H. Ju was supported by the Basic Science Program of the NRF of Korea (NRF-2022R1C1C1010052) and by the Basic Research Laboratory (grant MSIT no. RS-202400414849). 
We would also like to thank Hyun-Min Kim, Yeongrak Kim, Taehyeong Kim and Eric Dolores-Cuenca for their insightful comments.

\section{Preliminaries} \label{Sec:Preliminaries}

In this section we introduce some notions that are needed to follow this paper, such as the tensor-tensor multiplication with linear maps, tensor rank, and computational complexity. Before proceeding with the definition of these notions, we first introduce notations that we will adopt in this work. Throughout this paper, the base field $\mathbb{K}$ is either the real number field $\mathbb{R}$ or the complex number field $\mathbb{C}$. If a matrix $M$ is given, we let $M^{H}$ denote its conjugate transpose if $\mathbb{K}=\mathbb{C}$, or its transpose if $\mathbb{K}=\mathbb{R}$. Let $I$ denote the identity matrix. We will write $I_n$ for the $n \times n$ identity matrix when we want to emphasize the size. For a positive integer $n$, let $[n]$ denote the set $\{1,2,...,n\}$. For a vector space $V$, let $V^*$ denote its dual space.

In this research we will only deal with tensors of order $\leq 3$. As usual, the space of first-order tensors, that is, the space of vectors, of size $m$ is denoted by $\mathbb{K}^m$. The space of second-order tensors, that is, the space of matrices, of size $m \times n$ is denoted by $\mathbb{K}^{m \times n}$. Finally, the space of third-order tensors of size $m \times n \times p$ is denoted by $\mathbb{K}^{m \times n \times p}$.

Considering $\mathcal{A}=[a_{ijk}] \in \mathbb{K}^{m \times n \times p}$ as a stack of $p$ matrices in $\mathbb{K}^{m \times n}$, these $m \times n$ matrices are called \emph{frontal slices of $\mathcal{A}$}. We let $\mathcal{A}^{(i)}$ denote the $i$-th frontal slice of $\mathcal{A}$ and we denote $\mathcal{A}=\left[\begin{array}{c|c|c|c}
     \mathcal{A}^{(1)}& \mathcal{A}^{(2)}&\cdots&\mathcal{A}^{(p)}
\end{array}\right]$. In particular, the vectors $\left[\begin{array}{c|c|c|c}
    a_{ij1} & a_{ij2} & \cdots & a_{ijp}  
\end{array}\right]$ are called  \emph{mode-$3$ fibers of $\mathcal{A}$} for all $i \in [m],~ j \in [n]$. On the other hand, the slices $\mathcal{A}(:,j,:) \in \mathbb{K}^{m \times 1 \times p}$ are called \emph{lateral slices} of $\mathcal{A}$.

\subsection{Tensor-tensor multiplication with linear maps}\label{Tensor-tensor multiplication with linear maps}

We will follow the definitions of mode product, face-wise product, and tensor-tensor multiplication in \cite{MR3394164}. We use the terminology tensor-tensor ``multiplication"  rather than ``product" to distinguish it from ``tensor product" (of modules). The tensor-tensor multiplication in \cite{MR3394164} was defined only with arbitrary invertible linear maps, but it makes sense to define it with arbitrary linear maps. In this paper, we deal with tensor-tensor multiplication with not only invertible linear maps but also surjective and injective linear maps. In fact, the tensor-tensor multiplication with surjective orthogonal maps was recently introduced in \cite{keegan24}. However, to the best of our knowledge, the injective case has rarely been researched. 

\begin{definition}
Let $\mathcal{A}=[a_{ijk}]\in \mathbb{K}^{m\times n \times p}$ and let $M\in\mathbb{K}^{q\times p}$. The \emph{mode-$3$ product of $\mathcal{A}$ and $M$}, denoted as $\mathcal{A}{\times}_3 M$, gives a tensor in $\mathbb{K}^{m\times n \times q}$ whose entries are defined as
\begin{equation*}
(\mathcal{A}{\times}_3 M)_{ijk}=\sum\limits_{l=1}^{p}m_{kl}a_{ijl} ,
\end{equation*}
for $i\in[m]$, $j\in[n]$, and $k\in[q]$. 
\end{definition}

\begin{remark}
   We emphasize that the mode-$3$ product works exactly as the matrix-vector multiplication of a given matrix $M \in \mathbb{K}^{q \times p}$ and each mode-$3$ fiber of a given tensor $\mathcal{A} \in \mathbb{K}^{m \times n \times p}$. Furthermore, we can define the mode-$3$ product in a coordinate-free manner. Let $L:W \rightarrow W'$ be a linear map, and let $\mathcal{T}=\sum_{i=1}^r v_{i,1} \otimes v_{i,2} \otimes v_{i,3} \in U \otimes V \otimes W$, where $U,V,W$ and $W'$ are finite-dimensional vector spaces over $\mathbb{K}$. For ordered bases of $U,V,W$ and $W'$, let $M$ and $\mathcal{A}$ be, respectively, the matrix representing $L$ and the cube representing $\mathcal{T}$. Then, the mode-$3$ product $\mathcal{A} \times_3 M$ is the tensor representation of $L(\mathcal{A}):=\sum_{i=1}^r v_{i,1} \otimes v_{i,2} \otimes L(v_{i,3})$ with respect to the ordered bases. Because we deal with the computational aspect of tensor operations, we assume that the bases are fixed and handle matrices and cubes instead of abstract linear maps and multilinear maps. From this observation (or directly from the definition), we obtain that 
\begin{equation*}
    (\mathcal{A} \times_3 M) \times_3 N = \mathcal{A} \times_3 (NM)
\end{equation*}
for all $\mathcal{A} \in \mathbb{K}^{m \times n \times p}, M \in \mathbb{K}^{q \times p}$ and $N \in \mathbb{K}^{l \times q}$. 
\end{remark}

\begin{definition}
Let $\mathcal{A}\in \mathbb{K}^{m\times n \times p}$ and $\mathcal{B}\in \mathbb{K}^{n\times s \times p}$. The \emph{face-wise product of $\mathcal{A}$ and $\mathcal{B}$}, denoted as $\mathcal{A}\triangle\mathcal{B}$, gives a tensor in $\mathbb{K}^{m\times s \times p}$ whose frontal slices are
\begin{equation*}
(\mathcal{A}\triangle\mathcal{B})^{(i)}=\mathcal{A}^{(i)}\mathcal{B}^{(i)}
\end{equation*}
for $i\in [p]$, where the operation on the right-hand side is the matrix multiplication.
\end{definition}

Now we are ready to define the tensor-tensor multiplication with linear maps. As we already mentioned, we are going to define it with invertible, surjective, and injective linear maps. Since each tensor-tensor multiplication has slightly different properties depending on the type of linear map, we will introduce each definition and their properties separately.

\subsubsection{Case of invertible linear maps}\label{Sec:invertible}

Here, we introduce the definition of tensor-tensor multiplication with invertible linear maps, that first appeared in \cite{MR3394164}, and its properties.

\begin{definition}
Let $M \in \mathbb{K}^{p \times p}$ be invertible, and let $\mathcal{A}\in\mathbb{K}^{m\times n \times p}$ and $\mathcal{B}\in\mathbb{K}^{n\times s \times p}$. The \emph{tensor-tensor multiplication of $\mathcal{A}$ and $\mathcal{B}$ with $M$} or simply the \emph{$*_M$-product of $\mathcal{A}$ and $\mathcal{B}$}, denoted by $\mathcal{A}*_M\mathcal{B}$, gives a tensor in $\mathbb{K}^{m\times s\times p}$ that is defined as
\begin{equation*}
\mathcal{A}*_M\mathcal{B}=((\mathcal{A} \times_3 M)\triangle (\mathcal{B} \times_3 M))\times_3 M^{-1}.
\end{equation*}
\end{definition}

\begin{proposition}\label{Prop:InvertibleCase}
Let $M \in \mathbb{K}^{p \times p}$ be invertible. Then, 
\begin{itemize}
    \item [(\romannumeral1)] $(\mathbb{K}^{1 \times 1 \times p},*_M)$ forms a commutative ring with unity.
    \item [(\romannumeral2)] $(\mathbb{K}^{m \times 1 \times p},*_M)$ is a finitely generated free module over $(\mathbb{K}^{1 \times 1 \times p},*_M)$.
    \item [(\romannumeral3)] $\mathcal{A} \in \mathbb{K}^{m \times n \times p}$ represents a linear map $(\mathbb{K}^{n \times 1 \times p},*_M) \rightarrow (\mathbb{K}^{m \times 1 \times p},*_M)$ between finitely generated free modules over $(\mathbb{K}^{1 \times 1 \times p},*_M)$.
    
    \item [(\romannumeral4)] For $\mathcal{A} \in \mathbb{K}^{m \times n \times p}$ and $\mathcal{B} \in \mathbb{K}^{n \times s \times p}$, $\mathcal{A} *_M \mathcal{B} \in \mathbb{K}^{m \times s \times p}$ represents the composition of the linear maps $\mathcal{A},\mathcal{B}$ between finitely generated free modules over $(\mathbb{K}^{1 \times 1 \times p},*_M)$.
\end{itemize}
\end{proposition}
\begin{proof}
     For the detailed proofs of (\romannumeral1)-(\romannumeral2) we refer to \cite{MR3394164}, which also includes a proof of the associativity of the $*_M$-product in $\mathbb{K}^{m \times n \times p}$. 

(\romannumeral3) Let $\mathcal{A} \in \mathbb{K}^{m \times n \times p}$, $P,Q  \in \mathbb{K}^{n \times 1 \times p}$ and $v \in \mathbb{K}^{1 \times 1 \times p}$. From the definition of the $*_M$-product, it is obvious that $\mathcal{A} *_M (P+Q)=\mathcal{A} *_M P+\mathcal{A} *_M Q$. Since the associativity property of the $*_M$-product for third-order tensors was proved in \cite[Proposition 4.2]{MR3394164}, then $\mathcal{A} *_M (P *_M v)=(\mathcal{A} *_M P) *_M v$ also holds.

(\romannumeral4) By the associativity of the $*_M$-product, $(\mathcal{A} *_M \mathcal{B}) *_M P=\mathcal{A} *_M (\mathcal{B} *_M P)$ for all $\mathcal{A} \in \mathbb{K}^{m \times n \times p}, \mathcal{B} \in \mathbb{K}^{n \times s \times p}$ and $P \in \mathbb{K}^{s \times 1 \times p}$. Thus, the assertion is clear.

\end{proof}

\subsubsection{Case of surjective linear maps}\label{Sec:surject}

Here, we introduce the definition of tensor-tensor multiplication with surjective linear maps and its properties, which were stated in \cite{keegan24} for surjective orthogonal linear maps. 

Instead of the inverse, we will use the \emph{Moore-Penrose inverse} or simply \emph{pseudoinverse} of a nonzero surjective linear map. Throughout this paper, a surjective linear map implies a nonzero surjective linear map which is not invertible unless otherwise stated. We assume the non-invertiblility just in order to distinguish from the case of invertible linear maps. Note that if $M$ is invertible, then its pseudoinverse will be the same as the inverse $M^{-1}$. For $q<p$, if a linear map $L:\mathbb{K}^p\to\mathbb{K}^q$ is surjective, then its matrix representation $M \in \mathbb{K}^{q \times p}$ is of full-rank. In this case, we simply say that $M$ is surjective. Then, the pseudoinverse of $M$, denoted by $M^{+}$, is 
\begin{equation*}
    M^+=M^H(MM^H)^{-1}
\end{equation*}
so it plays the role of right inverse of $M$, i.e., $MM^+=I_q$. 

\begin{definition}
Let $M \in \mathbb{K}^{q \times p}$ ($q<p$) be surjective, and let $\mathcal{A}\in\mathbb{K}^{m\times n \times p}$ and $\mathcal{B}\in\mathbb{K}^{n\times s \times p}$. The \emph{tensor-tensor multiplication of $\mathcal{A}$ and $\mathcal{B}$ with $M$} or simply the \emph{$*_M$-product of $\mathcal{A}$ and $\mathcal{B}$}, denoted by $\mathcal{A}*_M\mathcal{B}$, gives a tensor in $\mathbb{K}^{m\times s\times p}$ that is defined as
\begin{equation*}
\mathcal{A}*_M\mathcal{B}=((\mathcal{A} \times_3 M)\triangle (\mathcal{B} \times_3 M))\times_3 M^{+}.
\end{equation*}
\end{definition}

We emphasize that for $\mathcal{A}\in\mathbb{K}^{m\times n \times p}$ and $M \in \mathbb{K}^{q \times p}$ surjective, the tensor $\mathcal{A} \times_3 M$ has mode-$3$ fibers of size $q$, not $p$. Furthermore, we note that
\begin{equation*}
    (\mathcal{A} \times_3 M^+) \times M=\mathcal{A} \times_3 (MM^+)=\mathcal{A} \times_3 I_q=\mathcal{A},
\end{equation*}
but
\begin{equation*}
    (\mathcal{A} \times_3 M) \times M^+=\mathcal{A} \times_3 (M^+M)\neq \mathcal{A}
\end{equation*}
in general.

\begin{proposition}\label{prop:SurjectiveCase}
Let $M \in \mathbb{K}^{q \times p}$ ($q<p$) be surjective, and let $\mathbb{K}^{m \times 1 \times q} \times_3 M^+$ denote the image of the mode-$3$ product by $M^+$ on $\mathbb{K}^{m \times 1 \times q}$, that is,
\begin{equation*}
    \mathbb{K}^{m \times 1 \times q} \times_3 M^+=\{P \times_3 M^+~|~P \in \mathbb{K}^{m \times 1 \times q}\}.
\end{equation*}
Then,
\begin{itemize}
    \item [(\romannumeral1)] $(\mathbb{K}^{1 \times 1 \times p},*_M)$ forms a commutative ring.
    \item [(\romannumeral2)] $(\mathbb{K}^{1 \times 1 \times q} \times_3 M^+,*_M)$ forms a commutative ring with unity.
    \item [(\romannumeral3)] $(\mathbb{K}^{m \times 1 \times q} \times_3 M^+,*_M)$ is a finitely generated free module over $(\mathbb{K}^{1 \times 1 \times p},*_M)$. In addition, this module is of dimension $q$.
    \item [(\romannumeral4)] $\mathcal{A} \in \mathbb{K}^{m \times n \times p}$ represents a linear map $(\mathbb{K}^{n \times 1 \times q} \times_3 M^+,*_M) \rightarrow (\mathbb{K}^{m \times 1 \times q} \times_3 M^+,*_M)$ between finitely generated free modules over $(\mathbb{K}^{1 \times 1 \times p},*_M)$.
    \item [(\romannumeral5)] For $\mathcal{A} \in \mathbb{K}^{m \times n \times p}$ and $\mathcal{B} \in \mathbb{K}^{n \times s \times p}$, $\mathcal{A} *_M \mathcal{B} \in \mathbb{K}^{m \times s \times p}$ represents the composition of the linear maps $\mathcal{A},\mathcal{B}$ between finitely generated free
    modules over $(\mathbb{K}^{1 \times 1 \times p},*_M)$.
\end{itemize}
\end{proposition}
\begin{proof} Note that the first assertion can be proved similarly as in \cite{keegan24}, but we are going to provide a detailed proof because comparing notations can be confusing.

 (\romannumeral1) We show associativity, distributivity and commutativity.
    
First, we prove associativity for third-order tensors of compatible size. Let $\mathcal{A} \in \mathbb{K}^{m \times n \times p}, \mathcal{B} \in \mathbb{K}^{n \times s \times p}$ and $\mathcal{C} \in \mathbb{K}^{s \times r \times p}$. Then,
        \begin{equation*}
            \begin{aligned}
            (\mathcal{A} *_M \mathcal{B})*_M \mathcal{C}&=(((((\mathcal{A} \times_3 M) \triangle (\mathcal{B} \times_3 M))\times_3 M^+)\times_3 M) \triangle (\mathcal{C} \times_3 M))\times_3 M^+\\
            &=((((\mathcal{A} \times_3 M) \triangle (\mathcal{B} \times_3 M))\times_3 MM^+) \triangle (\mathcal{C} \times_3 M))\times_3 M^+\\
            &=((((\mathcal{A} \times_3 M) \triangle (\mathcal{B} \times_3 M))\times_3 I_q) \triangle (\mathcal{C} \times_3 M))\times_3 M^+\\
            &=(((\mathcal{A} \times_3 M) \triangle (\mathcal{B} \times_3 M)) \triangle (\mathcal{C} \times_3 M))\times_3 M^+\\
            &=((\mathcal{A} \times_3 M) \triangle ((\mathcal{B} \times_3 M) \triangle (\mathcal{C} \times_3 M))\times_3 M^+\\
            &=\mathcal{A} *_M (\mathcal{B} *_M \mathcal{C}).
            \end{aligned}
        \end{equation*}

The distributivity and commutativity properties for $(\mathbb{K}^{1 \times 1 \times p}, *_M)$ are clear, since the face-wise product satisfies distributivity and commutativity for $\mathbb{K}^{1 \times 1 \times p}$.
        
(\romannumeral2) It is easy to check that $(\mathbb{K}^{1 \times 1 \times q}\times_3 M^+, *_M)$ is a subring of  $(\mathbb{K}^{1 \times 1 \times p},*_M)$. Thus, it is sufficient to show the existence of unity. Let $u=\widehat{u} \times_3 M^+ \in \mathbb{K}^{1 \times 1 \times p}$ where $\widehat{u} \in \mathbb{K}^{1 \times 1 \times q}$ is the tensor whose frontal slices are $I_1$. Then $u$ is the unity in $(\mathbb{K}^{1 \times 1 \times q}\times_3 M^+,*_M)$.

 (\romannumeral3) As in the proof of \cite[Proposition 4.1]{MR3394164}, let $\mathcal{B}_j$ be the $j$-th lateral slice of $\mathcal{I}:=\widehat{\mathcal{I}} \times_3 M^+$, where $\widehat{\mathcal{I}} \in \mathbb{K}^{m \times m \times q}$ denotes the tensor whose frontal slices are $I_m$. For arbitrary $Q \in \mathbb{K}^{m \times 1 \times q} \times_3 M^+$, there is a $P \in \mathbb{K}^{m \times 1 \times q}$ such that $Q=P \times_3 M^+$. Then,
        \begin{equation}\label{eq:SurjGenerating}
            \mathcal{I} *_M Q= (\widehat{\mathcal{I}} \triangle P) \times_3 M^+=P \times_3 M^+=Q.
        \end{equation}
This means that any arbitrary $Q\in \mathbb{K}^{m \times 1 \times q} \times_3 M^+$ can be expressed as a linear combination of $\mathcal{B}_j$'s over $\mathbb{K}^{1 \times 1 \times p}$.  Furthermore, (\ref{eq:SurjGenerating}) says that $\mathcal{I} *_M Q=0$ for $Q \in \mathbb{K}^{m \times 1 \times q} \times_3 M^+$ if and only if $Q=0$. Thus, $\{\mathcal{B}_j~|~j \in [q]\}$ is a basis of $(\mathbb{K}^{m \times 1 \times q} \times_3 M^+,*_M)$, and so this module is of dimension $q$.
        
(\romannumeral4)-(\romannumeral5) They are done similarly as in the proof of Proposition \ref{Prop:InvertibleCase}.
\end{proof}

\subsubsection{Case of injective linear maps}\label{Sec:injective}

As in the case of invertible or surjective maps, we introduce the definition of tensor-tensor multiplication with injective linear maps and its properties. Here we will also use the \emph{Moore-Penrose inverse} or simply \emph{pseudoinverse} of an injective linear map, and throughout this paper an injective linear map implies a nonzero injective linear map which is not invertible, unless otherwise stated. For $q>p$, if a linear map $L:\mathbb{K}^p\to\mathbb{K}^q$ is injective, then its matrix representation $M \in \mathbb{K}^{q \times p}$ is of full-rank, and we simply say that $M$ is injective in this case. Then, the pseudoinverse of $M$, denoted by $M^{+}$, is 
\begin{equation*}
    M^+=(M^HM)^{-1}M^H,
\end{equation*}
so it plays the role of left inverse of $M$, i.e., $M^+M=I_p$.

\begin{definition}
Let $M \in \mathbb{K}^{q \times p}$ ($q>p$) be injective, and let $\mathcal{A}\in\mathbb{K}^{m\times n \times p}$ and $\mathcal{B}\in\mathbb{K}^{n\times s \times p}$. The \emph{tensor-tensor multiplication of $\mathcal{A}$ and $\mathcal{B}$ with $M$} or simply the \emph{$*_M$-product of $\mathcal{A}$ and $\mathcal{B}$}, denoted by $\mathcal{A}*_M\mathcal{B}$, gives a tensor in $\mathbb{K}^{m\times s\times p}$ that is defined as
\begin{equation*}
\mathcal{A}*_M\mathcal{B}=((\mathcal{A} \times_3 M)\triangle (\mathcal{B} \times_3 M))\times_3 M^{+}.
\end{equation*}
\end{definition}

As in the case of surjective maps, we emphasize that for $\mathcal{A}\in\mathbb{K}^{m\times n \times p}$ and $M \in \mathbb{K}^{q \times p}$ injective, the tensor $\mathcal{A} \times_3 M$ has mode-$3$ fibers of size $q$, not $p$. Furthermore, we note that
\begin{equation*}
    (\mathcal{A} \times_3 M) \times M^+=\mathcal{A} \times_3 (M^+M)=\mathcal{A} \times_3 I_p=\mathcal{A},
\end{equation*}
but
\begin{equation*}
    (\mathcal{A} \times_3 M^+) \times M=\mathcal{A} \times_3 (MM^+)\neq \mathcal{A}
\end{equation*}
in general.

Unlike the invertible or surjective cases, $(\mathbb{K}^{1 \times 1 \times p},*_M)$ is not a ring. The following example shows that in general it does not satisfy associativity. Therefore, it is difficult to argue the module structure as we did on the invertible or surjective cases.
\begin{example}\label{ex:InjectiveNotAssociative}
    Let $M=\begin{bmatrix}
        1 & 1 & 1\\
        1 & -2 & 1\\
        1 & 1 & 0\\
        0 & 1 & 2
    \end{bmatrix}$, then $M^+=\frac{1}{95}\begin{bmatrix}
        29 & 23 & 43 & -26 \\
        13 & -29 & 16 & 8 \\
        4 & 13 & -17 & 39 \\
    \end{bmatrix}$. If
    \begin{equation*}
        u=\left[\begin{array}{c|c|c}
     1 & 1 & 1
\end{array}\right],~ v=\left[\begin{array}{c|c|c}
     1 & 2 & 3
\end{array}\right],~ w=\left[\begin{array}{c|c|c}
     -1 & 1 & 5
\end{array}\right] \in \mathbb{K}^{1 \times 1 \times 3},
    \end{equation*}
    then 
    \begin{equation*}
        (u *_M v) *_M w = \frac{1}{9025} \left[\begin{array}{c|c|c}
     -437016 & 307308 & 1033434
\end{array}\right]
    \end{equation*}
    and
    \begin{equation*}
        u *_M (v *_M w) = \frac{1}{9025} 
        \left[\begin{array}{c|c|c}
     -386472 & 312276 & 1008918
\end{array}\right]
    \end{equation*}
    are not the same.
\end{example}

\subsubsection{Identity, inverse, conjugate transpose, and f-diagonal tensor}\label{Sec:Identity,Inverse,etc}

We define the identity, inverse, conjugate transpose, and f-diagonal tensor for all the cases of $M$ being invertible, surjective, or injective. From now on, if we say that $M\in \mathbb{K}^{q \times p}$ is of full-rank, then it means that
\begin{equation*}
    M~\text{is}~ \begin{cases}
        \text{invertible} & \text{if $p=q$},\\
        \text{surjective} & \text{if $p>q$},\\
        \text{injective} & \text{if $p<q$}.
    \end{cases}
\end{equation*}

\begin{definition}
Let $M \in \mathbb{K}^{q \times p}$ be of full-rank.
\begin{itemize}
    \item [(\romannumeral1)] A tensor $\mathcal{I} \in \mathbb{K}^{n \times n \times p}$ such that all frontal slices of $\mathcal{I} \times_3 M$ are $I_n$ is called an \emph{$*_M$-identity tensor}. That is, for all $i \in [q]$,
    \begin{equation*}
        (\mathcal{I} \times_3 M)^{(i)}=I_n.
    \end{equation*}
    \item [(\romannumeral2)] For $\mathcal{A} \in \mathbb{K}^{n \times n \times p}$, its \emph{$*_{M}$-conjugate‐transpose}, denoted by $\mathcal{A}^H$, is defined so that, for all $i \in [q]$,
    \begin{equation*}
        (\mathcal{A}^H \times_3 M)^{(i)}=((\mathcal{A}\times_3 M)^{(i)})^H.
    \end{equation*}
    \item [(\romannumeral3)] A tensor $\mathcal{A} \in \mathbb{K}^{n \times n \times p}$ is said to be \emph{$*_M$-Hermitian} (or \emph{$*_M$-symmetric} if $\mathbb{K}=\mathbb{R}$) if $\mathcal{A}^H=\mathcal{A}$, i.e., the frontal slices $(\mathcal{A}\times_3 M)^{(i)}$ are Hermitian for all $i\in [q]$.
    \item [(\romannumeral4)] A tensor $\mathcal{S} \in \mathbb{K}^{n \times n \times p}$ is said to be \emph{f-diagonal} if its only nonzero elements are contained within its diagonal mode-$3$ fibers, or equivalently $\mathcal{S} \times_3 M$ has this property.
\end{itemize}
\end{definition}

If $M$ is invertible, then it has already been proved that the $*_M$-identity tensor plays the role of identity $\mathcal{I} \in \mathbb{K}^{n \times n \times p}$ with respect to the $*_M$-product, i.e., $\mathcal{A} *_M \mathcal{I}=\mathcal{I} *_M \mathcal{A}=\mathcal{A}$ for all $\mathcal{A} \in \mathbb{K}^{n \times n \times p}$, and such tensor is uniquely determined \cite{MR3394164}. However, if $M$ is surjective, then not only does the $*_M$-identity tensor not act in general as an identity but the $*_M$-identity tensor is also not uniquely determined in general \cite{keegan24}. We prove that if $M$ is injective, then the $*_M$-identity tensor acts as in the invertible case, if it exists.

\begin{lemma}\label{lemma:inj_identity}
Let $M \in \mathbb{K}^{q \times p}$ be injective and assume that there exists a $*_M$-identity tensor in $\mathbb{K}^{n \times n \times p}$.
\begin{itemize}
    \item [(\romannumeral1)] A $*_M$-identity tensor $\mathcal{I} \in \mathbb{K}^{n \times n \times p}$ is an identity with respect to $*_M$-product.
    \item [(\romannumeral2)] The $*_M$-identity tensor is uniquely determined.
\end{itemize} 
\end{lemma}

\begin{proof}
(\romannumeral1) For any $\mathcal{A}\in\mathbb{K}^{n\times n \times p}$,
\begin{equation*}
\mathcal{I} *_M\mathcal{A}= ((\mathcal{I} \times_3 M) \triangle (\mathcal{A}\times_3 M))\times M^+ = \mathcal{A} \times_3 (M^+M)=\mathcal{A} \times_3 I_p= \mathcal{A}. 
\end{equation*}
Similarly, $\mathcal{A} *_M\mathcal{I}=\mathcal{A}$.

(\romannumeral2) If $\mathcal{I}$ and $\mathcal{J}$ are $*_M$-identity tensors, then (\romannumeral1) implies that
    \begin{equation*}
        \mathcal{I}= \mathcal{I} *_M\mathcal{J}=\mathcal{J}.
    \end{equation*}
    Thus, the $*_M$-identity tensor is uniquely determined.

\end{proof}

\begin{remark}\label{rmk:RemarkIdentityTensor}
We note the following:
    \begin{itemize}
        \item [(\romannumeral1)] It is possible for the $*_M$-identity tensor to not exist when $M$ is injective. Note that the $*_M$-identity tensor exists if and only if $\left[\begin{array}{c|c|c|c}
   1 & 1 & \cdots & 1 
\end{array}\right]$ is in the image of $\times_3 M:\mathbb{K}^{1 \times 1 \times p} \rightarrow \mathbb{K}^{1 \times 1 \times q}$, but for instance the image of $\times_3 M$ for $M=\begin{bmatrix}
    1 & 1 \\
    0 & 0 \\
    1 & 0
\end{bmatrix}$ does not contain that vector.
        \item [(\romannumeral2)] When $M$ is injective, for each $\mathcal{A} \in \mathbb{K}^{n \times n \times p}$ it is  obvious that a $*_M$-conjugate transpose exists uniquely. 
    \end{itemize}
\end{remark}

\begin{definition}
Let $M \in \mathbb{K}^{q \times p}$ be of full-rank. Assume that an identity tensor $\mathcal{I} \in \mathbb{K}^{n \times n \times p}$ exists. A tensor $\mathcal{A}\in\mathbb{K}^{n\times n \times p}$ is said to be \textit{$*_M$-invertible} if there exists a tensor $\mathcal{X}\in\mathbb{K}^{n\times n \times p}$ such that
\begin{equation*}
\mathcal{A}*_M \mathcal{X}=\mathcal{X}*_M \mathcal{A} = \mathcal{I}.
\end{equation*}
\end{definition}

\begin{remark}\label{rmk:UniquenessOfInverse}
    If $M$ is surjective, then a tensor may have several inverses for the same reason that the identity tensor is not unique in general. However, obviously, if $M$ is either invertible or injective, the inverse is unique if it exists, and we denote such inverse as $\mathcal{A}^{-1}$.
\end{remark}

\subsubsection{Bilinear structure}\label{Sec:bilinear_structure}
In Propositions \ref{Prop:InvertibleCase}-\ref{prop:SurjectiveCase}, we looked into abstract algebraic structures of the $*_M$-product. Even though this structure also shows how the $*_M$-product is meaningful algebraically, it does not tell us anything about the computational aspect of the tensor-tensor multiplication. In this section, in order to analyze the computational aspect of the tensor-tensor multiplication, we verify its bilinearity over the field $\mathbb{K}$. This will make it possible to consider its tensor rank and asymptotic exponent, each of which corresponds to a measure of computational complexity. We refer to \cite{MR2865915} for algebraic and geometric approaches to tensors.

\begin{lemma}\label{lemma:bilinearity}
Let $M \in \mathbb{K}^{q \times p}$ be of full-rank. Then the tensor-tensor multiplication map with $M$, denoted by
\begin{equation*}
\operatorname{TM}_{\langle m,n,s \rangle,p,q}:\mathbb{K}^{m\times n \times p} \times \mathbb{K}^{n\times s \times p} \rightarrow \mathbb{K}^{m\times s \times p},
\end{equation*}
is a bilinear map over $\mathbb{K}$.
\end{lemma}
\begin{proof}
At first, note that the face-wise product \(\triangle\) is bilinear, and both $\times_3 M$ and $\times_3 M^+$ are linear. Let $\mathcal{A}, \mathcal{A}_1, \mathcal{A}_2 \in \mathbb{K}^{m\times n \times p}$, $\mathcal{B},\mathcal{B}_1, \mathcal{B}_2\in \mathbb{K}^{n\times s \times p}$, and $c \in \mathbb{K}$. Then,
\begin{equation*}
    \begin{aligned}
        &(c\mathcal{A}_1+\mathcal{A}_2) *_M \mathcal{B}\\
        &= (((c\mathcal{A}_1+\mathcal{A}_2) \times_3 M) \triangle (\mathcal{B}\times_3 M)) \times_3 M^+\\
        &=((c\mathcal{A}_1 \times_3 M +\mathcal{A}_2\times_3 M) \triangle (\mathcal{B}\times_3 M)) \times_3 M^+\\
        &=(c(\mathcal{A}_1 \times_3 M)\triangle (\mathcal{B}\times_3 M) +(\mathcal{A}_2\times_3 M) \triangle (\mathcal{B}\times_3 M)) \times_3 M^+\\
        &=c((\mathcal{A}_1 \times_3 M)\triangle (\mathcal{B}\times_3 M))\times_3 M^+ +((\mathcal{A}_2\times_3 M) \triangle (\mathcal{B}\times_3 M)) \times_3 M^+\\
        &=c(\mathcal{A}_1 *_M \mathcal{B}) + \mathcal{A}_2 *_M \mathcal{B}.
    \end{aligned}
\end{equation*}
Similarly, we can show that 
\begin{equation*}
    \mathcal{A} *_M (c\mathcal{B}_1+\mathcal{B}_2)=c(\mathcal{A}*_M \mathcal{B}_1)+\mathcal{A}*_M \mathcal{B}_2.
\end{equation*}
Therefore, $*_M$ is bilinear.
\end{proof}

Now, we describe the formula of this bilinear map by considering it as an element, as follows
\begin{equation*}
   \operatorname{TM}_{\langle m,n,s \rangle,p,q} \in (\mathbb{K}^{m \times n \times p})^* \otimes (\mathbb{K}^{n \times s \times p})^* \otimes \mathbb{K}^{m \times s \times p}.
\end{equation*} 
Let $M=[m_{ij}]\in \mathbb{K}^{q \times p}$ be of full-rank, and let $M^+=[n_{ij}] \in \mathbb{K}^{p \times q}$. We start by considering the formula of the face-wise product $\triangle:\mathbb{K}^{m \times n \times p} \times \mathbb{K}^{n \times s \times p} \rightarrow \mathbb{K}^{m \times s \times p}$. Assume that $M$ is a matrix representation with respect to the ordered bases $\{e_1,...,e_p\}$ of $\mathbb{K}^{p}$ and $\{\tilde{e}_1,...,\tilde{e}_q\}$ of $\mathbb{K}^q$. Let $\mathbb{K}^m, \mathbb{K}^n$ and $\mathbb{K}^s$ have the ordered bases $\{u_1,...,u_m\},\{v_1,...,v_n\}$ and $\{w_1,...,w_s\}$, respectively. We let $\{e_1^*,...,e_p^*\}$ be the dual basis of $\{e_1,...,e_p\}$ and similarly for the other bases. Then
\begin{equation}\label{eq:FaceWiseProductTensor}
    \triangle=\sum_{r=1}^q\sum_{i=1}^m\sum_{j=1}^n\sum_{k=1}^s (u_i^* \otimes v_j^* \otimes \tilde{e}_r^*)\otimes (v_j^* \otimes w_k^* \otimes \tilde{e}_r^*) \otimes (u_i \otimes w_k \otimes \tilde{e}_r).
\end{equation}
Note that for each fixed $r$, the summand is just the matrix multiplication tensor $\operatorname{M}_{\langle m,n,s \rangle}$, which multiplies $m \times n$ and $n \times s$ matrices. That is, $\triangle$ is the direct sum of $q$ copies of this matrix multiplication tensor, i.e., 
\begin{equation*}
    \triangle=\bigoplus_{r=1}^q\operatorname{M}_{\langle m,n,s \rangle},
\end{equation*}
where the direct sum of two tensors $T \in V_1 \otimes \cdots \otimes V_d$ and $T' \in W_1 \otimes \cdots \otimes W_d$ is the tensor $T+T' \in (V_1 \oplus W_1) \otimes \cdots \otimes (V_d \oplus W_d)$.
Recall that $M^{\top}$ represents the dual map $(\mathbb{K}^q)^* \rightarrow (\mathbb{K}^p)^*$ of $M$ with respect to the dual bases. In the definition of the $*_M$-product, we first perform the mode-$3$ product of $M$ to all mode-$3$ fibers of the given two tensors, so the term $\tilde{e}_r^*$ at the first two factors of (\ref{eq:FaceWiseProductTensor}) will be changed to 
\begin{equation*}
    M^{\top}\tilde{e}_r^*=m_{r1}e_1^*+m_{r2}e_2^*+\cdots+m_{rp}e_p^*.
\end{equation*}
Indeed, in the definition of the $*_M$-product, right after performing $\triangle$, we carry out the mode-$3$ product of $M^+$ to all mode-$3$ fibers of the resulting tensor, so the term $\tilde{e}_r$ at the last factor of (\ref{eq:FaceWiseProductTensor}) will be changed to 
\begin{equation*}
    M^+\tilde{e}_r=n_{1r}e_1+n_{2r}e_2+\cdots+n_{pr}e_p.
\end{equation*}
Therefore,
\begin{equation}\label{eq:TMTensor}
\begin{aligned}
    &\operatorname{TM}_{\langle m,n,s \rangle,p,q}\\
    &\begin{aligned}
        =\sum_{r=1}^q\sum_{i=1}^m\sum_{j=1}^n\sum_{k=1}^s (u_i^* \otimes v_j^* \otimes M^{\top}\tilde{e}_r^*)\otimes (v_j^* \otimes w_k^* \otimes M^{\top}\tilde{e}_r^*)\\ \otimes (u_i \otimes w_k \otimes M^+\tilde{e}_r)
    \end{aligned} \\
    &\begin{aligned}
        =\sum_{r=1}^q\sum_{i=1}^m\sum_{j=1}^n\sum_{k=1}^s& (u_i^* \otimes v_j^* \otimes (m_{r1}e_1^*+\cdots+m_{rp}e_p^*))\\
        &\otimes (v_j^* \otimes w_k^* \otimes (m_{r1}e_1^*+\cdots+m_{rp}e_p^*)) \\
        &\otimes (u_i \otimes w_k \otimes (n_{1r}e_1+\cdots+n_{pr}e_p))
    \end{aligned}\\
    &\in (\mathbb{K}^{m \times n \times p})^* \otimes (\mathbb{K}^{n \times s \times p})^* \otimes \mathbb{K}^{m \times s \times p}.
\end{aligned}
\end{equation}
Case 1) If $p=q$ (i.e., $M$ is invertible), then each of $M^{\top}$ and $M^+=M^{-1}$ maps a basis to a basis, and so
\begin{equation}\label{eq:TMTensorInvertible}
    \operatorname{TM}_{\langle m,n,s \rangle,p}:=\operatorname{TM}_{\langle m,n,s \rangle,p,p}=\bigoplus_{r=1}^{p} \operatorname{M}_{m,n,s}.
\end{equation}
Case 2) If $p>q$ (i.e., $M$ is surjective), then both $M^{\top}$ and $M^+$ are injective, and so 
\begin{equation}\label{eq:TMTensorSurjective}
    \operatorname{TM}_{\langle m,n,s \rangle,p,q}=\bigoplus_{r=1}^{q} \operatorname{M}_{m,n,s}.
\end{equation}
Case 3) If $p<q$ (i.e., $M$ is injective), then both $M^{\top}$ and $M^+$ are surjective. In this case, both $\{M^{\top}\tilde{e}_r^*~|~r \in [q]\}$ and $\{M^{+}\tilde{e}_r~|~r \in [q]\}$ are linearly dependent sets, so some $q-p$ vectors on each set are linear combinations of the other linearly independent $p$ vectors. Thus, for some $c_r \in \mathbb{K}$,
\begin{equation}\label{eq:TMTensorInjective}
    \operatorname{TM}_{\langle m,n,s \rangle,p,q}=\bigoplus_{r=1}^{p} c_r\operatorname{M}_{m,n,s}.
\end{equation}
Notably, if $p<q$, then $\operatorname{TM}_{\langle m,n,s \rangle,p,q}$ can be the zero map. See the following example.

\begin{example}\label{ex:ZeroTMTensor}
    Let $p=1$ and $q=2$, and let $M=\begin{bmatrix}
        1 \\ -1
    \end{bmatrix}$. Then $M^{\top}=\begin{bmatrix}
        1 & -1
    \end{bmatrix}$ and $M^+=\frac{1}{2}\begin{bmatrix}
        1 & -1
    \end{bmatrix}$. Thus,
    \begin{equation*}
\begin{aligned}
    &\operatorname{TM}_{\langle m,n,s \rangle,p,q}\\
    &=\sum_{i=1}^m\sum_{j=1}^n\sum_{k=1}^s (u_i^* \otimes v_j^* \otimes M^{\top}\tilde{e}_1^*)\otimes (v_j^* \otimes w_k^* \otimes M^{\top}\tilde{e}_1^*) \otimes (u_i \otimes w_k \otimes M^+\tilde{e}_1)\\
    &~+\sum_{i=1}^m\sum_{j=1}^n\sum_{k=1}^s (u_i^* \otimes v_j^* \otimes M^{\top}\tilde{e}_2^*)\otimes (v_j^* \otimes w_k^* \otimes M^{\top}\tilde{e}_2^*) \otimes (u_i \otimes w_k \otimes M^+\tilde{e}_2)\\
    &=\frac{1}{2}\sum_{i=1}^m\sum_{j=1}^n\sum_{k=1}^s (u_i^* \otimes v_j^* \otimes e_1)\otimes (v_j^* \otimes w_k^* \otimes e_1) \otimes (u_i \otimes w_k \otimes e_1)\\
    &~+\frac{1}{2}\sum_{i=1}^m\sum_{j=1}^n\sum_{k=1}^s (u_i^* \otimes v_j^* \otimes (-e_1))\otimes (v_j^* \otimes w_k^* \otimes (-e_1)) \otimes (u_i \otimes w_k \otimes (-e_1))\\
    &=0.
\end{aligned}
\end{equation*}
For example, when $m=n=s=2$, if $A=\begin{bmatrix}
    a_{11} & a_{12}\\
    a_{21} & a_{22}
\end{bmatrix}$ and $B=\begin{bmatrix}
    b_{11} & b_{12}\\
    b_{21} & b_{22}
\end{bmatrix}$, then
\begin{equation*}
\begin{aligned}
    &A *_M B\\
    &=((A \times_3 M) \triangle (B \times_3 M))\times_3 M^+\\
    &=\left(\begin{bmatrix}
    \begin{matrix}
        a_{11} & a_{12}\\
    a_{21} & a_{22}
    \end{matrix} \mathrel{\bigg|} \begin{matrix}
        -a_{11} & -a_{12}\\
    -a_{21} & -a_{22}
    \end{matrix}
\end{bmatrix}\triangle \begin{bmatrix}
    \begin{matrix}
        b_{11} & b_{12}\\
    b_{21} & b_{22}
    \end{matrix}\mathrel{\bigg|} \begin{matrix}
        -b_{11} & -b_{12}\\
    -b_{21} & -b_{22}
    \end{matrix}
\end{bmatrix}\right)\times_3 M^+\\
&=\left(\begin{bmatrix}
    \begin{matrix}
        a_{11}b_{11}+a_{12}b_{21} & a_{11}b_{12}+a_{12}b_{22}\\
    a_{21}b_{11}+a_{22}b_{21} & a_{21}b_{12}+a_{22}b_{22}
    \end{matrix}\mathrel{\bigg|} \begin{matrix}
       \begin{matrix}
        a_{11}b_{11}+a_{12}b_{21} & a_{11}b_{12}+a_{12}b_{22}\\
    a_{21}b_{11}+a_{22}b_{21} & a_{21}b_{12}+a_{22}b_{22}
    \end{matrix}
    \end{matrix}
\end{bmatrix}\right)\times_3 M^+\\
&=\begin{bmatrix}
    0 & 0\\
    0 & 0
\end{bmatrix}.
\end{aligned}
\end{equation*}
\end{example}

\subsection{Tensor rank and computational complexity}\label{Sec:Tensor rank}

Since we have showed that the tensor-tensor multiplication is a bilinear map, i.e., a tensor, we can think of its tensor rank and asymptotic exponent. In this section, we introduce basic notions in algebraic complexity theory including the definitions of tensor rank, total complexity, and asymptotic exponent.

Although the notion of tensor rank can be extended to all multilinear maps, we will define it for only bilinear maps since this paper deals exclusively with at most bilinear maps. 

\begin{definition}
    For a bilinear map $T \in (\mathbb{K}^m)^* \otimes (\mathbb{K}^n)^* \otimes \mathbb{K}^s$, its \emph{tensor rank} or \emph{bilinear complexity}, denoted by $\mathbf{R}(T)$ is defined as
    \begin{equation*}
        \mathbf{R}(T)=\min\left\{r~\middle|~T=\sum_{i=1}^r \alpha_i \otimes \beta_i \otimes v_i,~\alpha_i \in (\mathbb{K}^m)^*, \beta_i \in (\mathbb{K}^n)^*,v_i \in \mathbb{K}^s,~\text{for}~i \in [r] \right\}.
    \end{equation*}
\end{definition}

From the definition, a formula of $T$ gives us an upper bound of the tensor rank $\mathbf{R}(T)$. On the other hand, from various \emph{flattenings} (see \cite{MR2865915}), we can obtain a lower bound of the tensor rank $\mathbf{R}(T)$. In this aspect, the notion of conciseness is important. It is possible to write the definition of conciseness for all multilinear maps, but we only introduce it for bilinear maps. Note that when a bilinear map $T \in (\mathbb{K}^m)^* \otimes (\mathbb{K}^n)^* \otimes \mathbb{K}^s$ is given, we can think of three linear maps derived from it:
\begin{equation}\label{eq:flattenings}
\begin{aligned}
    &T_{\mathbb{K}^m}:\mathbb{K}^m \rightarrow (\mathbb{K}^n)^* \otimes \mathbb{K}^s;\\
    &T_{\mathbb{K}^n}:\mathbb{K}^n \rightarrow (\mathbb{K}^m)^* \otimes \mathbb{K}^s;~\text{and}\\
    &T_{(\mathbb{K}^s)^*}:(\mathbb{K}^s)^* \mapsto (\mathbb{K}^m)^* \otimes (\mathbb{K}^n)^*.
\end{aligned}
\end{equation} 

\begin{definition}
    A bilinear map $T \in (\mathbb{K}^m)^* \otimes (\mathbb{K}^n)^* \otimes \mathbb{K}^s$ is said to be \emph{concise} if all the three linear maps at (\ref{eq:flattenings}) are injective (possibly invertible).
\end{definition}

\begin{remark}\label{rmk:LowerBoundFromConcise}
It is well-known that if a bilinear map $T \in (\mathbb{K}^m)^* \otimes (\mathbb{K}^n)^* \otimes \mathbb{K}^s$ is concise then $\mathbf{R}(T) \geq \max\{m,n,s\}$ \cite{MR2865915}. 
Note that the matrix multiplication tensor $\operatorname{M}_{\langle m,n,s \rangle}$ is concise.\\
Case 1) If $M \in \mathbb{K}^{q \times p}$ is invertible (i.e., $p=q$), then (\ref{eq:TMTensorInvertible}) implies that 
\begin{equation*}
    \operatorname{TM}_{\langle m,n,s \rangle,p}=\bigoplus_{r=1}^{p} \operatorname{M}_{m,n,s}\in (\mathbb{K}^{m \times n \times p})^* \otimes (\mathbb{K}^{n \times s \times p})^* \otimes \mathbb{K}^{m \times s \times p}
\end{equation*}
is concise too, and so
\begin{equation*}
    \mathbf{R}(\operatorname{TM}_{\langle m,n,s \rangle,p}) \geq \max\{mnp,nsp,msp\}.
\end{equation*}
Case 2) Assume that $M \in \mathbb{K}^{q \times p}$ is surjective (i.e., $p>q$). Then (\ref{eq:TMTensorInvertible}) implies that
\begin{equation*}
    \operatorname{TM}_{\langle m,n,s \rangle,p,q}=\bigoplus_{r=1}^{q} \operatorname{M}_{m,n,s}\in (\mathbb{K}^{m \times n \times p})^* \otimes (\mathbb{K}^{n \times s \times p})^* \otimes \mathbb{K}^{m \times s \times p}
\end{equation*}
is not concise. However, we can easily consider this tensor as a concise tensor in a smaller space. Let $V_1^*,V_2^*$ and $V_3$ be the $\mathbb{K}$-vector spaces with the bases $\{u_i^* \otimes v_j^* \otimes M^{\top}\tilde{e}_r^*~|~i \in [m], j \in [n], r \in [q] \}$, $\{v_j^* \otimes w_k^* \otimes M^{\top}\tilde{e}_r^*~|~j \in [n], k \in [s], r \in [q] \}$, and $\{u_i^* \otimes w_k^* \otimes M^{+}\tilde{e}_r~|~i \in [m], k \in [s], r \in [q] \}$, respectively. Now, consider their tensor product. Then, as an element in the smaller space,
\begin{equation*}
    \operatorname{TM}_{\langle m,n,s \rangle,p,q}=\bigoplus_{r=1}^{q} \operatorname{M}_{m,n,s} \in V_1^* \otimes V_2^* \otimes V_3
\end{equation*}
is concise, and so
\begin{equation*}
    \mathbf{R}(\operatorname{TM}_{\langle m,n,s \rangle,p,q}) \geq \max\{mnq,nsq,msq\}.
\end{equation*}
Case 3) In the case of injective $M \in \mathbb{K}^{q \times p}$, Example \ref{ex:ZeroTMTensor} shows that we cannot obtain any nonzero lower bound of  $\mathbf{R}(\operatorname{TM}_{\langle m,n,s \rangle,p,q})$ in general.  
\end{remark}

The tensor rank measures the minimal number of bilinear multiplications to calculate a bilinear map. It is also natural to measure the minimal number of additions, subtractions, multiplications, divisions, and constant multiplications, and this is measured by the \emph{total complexity}. Even for more general  rational maps $\varphi=(\varphi_1,...,\varphi_n):\mathbb{K}^m \rightarrow \mathbb{K}^n$, we can define their total complexity as the total complexity of the set of rational functions $F:=\{f_1,...,f_n\}$. We refer to \cite[Sections 4.1]{MR1440179} for the precise definition of the total complexity. For a rational map $\varphi$, we let $\mathbf{L}(\varphi)$ denote the total complexity of $\varphi$.

Now, we define two kinds of asymptotic exponents, one is on the tensor rank (of bilinear maps) and the other is on the total complexity (of bilinear maps or rational maps more generally).

\begin{definition}
Let $(a_n)$ denote a sequence $(a_1,a_2,a_3,\dots)$.
\begin{itemize}
    \item [(\romannumeral1)] For $(T_n)$ a sequence of bilinear maps, the \emph{asymptotic exponent of $(T_n)$ on the tensor rank}, denoted by $\omega(T_n)$, is defined by
    \begin{equation*}
        \omega(T_n)=\inf\{\tau \ | \ \mathbf{R}(T_n) = O(n^\tau)\}.
    \end{equation*}
    \item [(\romannumeral2)] For $(\varphi_n)$ a sequence of rational maps, the \emph{asymptotic exponent of $(\varphi_n)$ on the total complexity}, denoted by $\widehat{\omega}(\varphi_n)$, is defined by
    \begin{equation*}
        \widehat{\omega}(\varphi_n)=\inf\{\tau \ | \ \mathbf{L}(\varphi_n) = O(n^\tau)\}.
    \end{equation*}
\end{itemize}
\end{definition}
Note that it is well known that $\omega(\operatorname{M}_{\langle n\rangle})=\widehat{\omega}(\operatorname{M}_{\langle n \rangle})$ \cite{MR1440179, blaeser13}.

\section{Algebraic tensor functions}\label{Sec:algebraic}

We simply write $\operatorname{TM}_{\langle n \rangle,p,q}$ for $\operatorname{TM}_{\langle n,n,n\rangle, p,q}$ and $\operatorname{M}_{\langle n \rangle}$ for $\operatorname{M}_{\langle n,n,n \rangle}$. In this section, we introduce some results on the asymptotic exponent of $\operatorname{TM}_{\langle n \rangle,p,q}$ on its tensor rank, as well as on its total complexity.

\subsection{Asymptotic exponent of the tensor-tensor multiplication}
We claim that for the invertible and surjective case
\begin{equation*}
\omega(\operatorname{TM}_{\langle n\rangle,p,q})=\omega(\operatorname{M}_{\langle n \rangle}).
\end{equation*}
In order to prove this claim, we will use the following lemma which was stated in \cite{MR3749414}, that generalizes \cite[Lemma 14.10]{MR1440179}.

\begin{lemma}\label{commute}
Let $U,V,W$ and $U',V',W'$ be finite dimensional $\mathbb{K}$-vector spaces, and let $T:U \times V \rightarrow W$ and $T':U' \times V' \rightarrow W'$ be bilinear maps. If there is an injective (possibly invertible) map $f:U \times V \rightarrow U' \times V'$ and a surjective map $g:W' \rightarrow W$ such that the diagram
\[
\begin{tikzcd}
    U \times V \arrow[d,"T"]  \arrow[r,"f"] &  U' \times V' \arrow[d, "T'"] \\
    W &    W' \arrow[l,"g"]
\end{tikzcd}
\]
commutes, then $\mathbf{R}(T) \leq \mathbf{R}(T')$.
\end{lemma}

\begin{theorem}\label{thm:AsymptoticExponentRank}
If $M \in \mathbb{K}^{q \times p}$ is either invertible or surjective (i.e., $p \geq q$), then
\begin{equation}\label{eq:AsymptoticExponentRank}
    \omega(\operatorname{TM}_{\langle n \rangle,p,q})=\omega(\operatorname{M}_{\langle n \rangle}).
\end{equation}
\end{theorem}
\begin{proof}
From (\ref{eq:TMTensorInvertible}) and (\ref{eq:TMTensorSurjective}), we obtain that $\mathbf{R}(\operatorname{TM}_{\langle n \rangle,p,q})\leq \mathbf{R}(\operatorname{M}_{\langle n \rangle})\cdot q$. Hence,
\begin{equation*}
    \omega(\operatorname{TM}_{\langle n \rangle,p,q})\leq \omega(\operatorname{M}_{\langle n \rangle}).
\end{equation*}
Now, we prove the other inequality. We have the following commutative diagram:
\[
\begin{tikzcd}
(A, B) \arrow[d, phantom]  \arrow[r, phantom, "\longmapsto"] & (\mathcal{A},\mathcal{B})\arrow[d, phantom] \arrow[r, phantom, "\longmapsto"] & (\mathcal{A}\times_{3}M^{+},\mathcal{B}\times_{3}M^{+}) \\
    \mathbb{K}^{n\times n}\times\mathbb{K}^{n\times n} \arrow[r,"f_1"] \arrow[d,"\operatorname{M}_{\langle n \rangle}"] & \mathbb{K}^{n\times n \times q}\times\mathbb{K}^{n\times n\times q} \arrow[r,"f_2"] \arrow[d, phantom, "\circlearrowright"]  & \mathbb{K}^{n\times n\times p}\times\mathbb{K}^{n\times n\times p} \arrow[d,"\operatorname{TM}_{\langle n \rangle,p,q}"] \\
    \mathbb{K}^{n\times n} \arrow[d, phantom] & \mathbb{K}^{n\times n\times q} \arrow[l,"g_2"] \arrow[d, phantom] & \mathbb{K}^{n\times n\times p} \arrow[l,"g_1"]  \arrow[d, phantom] \\
    AB & \mathcal{A}\triangle\mathcal{B} \arrow[l, phantom, "\longleftarrow\!\shortmid"] & (\mathcal{A}\triangle\mathcal{B})\times_{3}M^{+} \arrow[l, phantom, "\longleftarrow\!\shortmid"]
\end{tikzcd}
\]
where $f_1,f_2, g_1$, and $g_2$ are respectively defined by
\begin{itemize}
    \item $f_1(A,B)=\left(\left[\begin{array}{c|c|c|c}
   A & O & \cdots & O 
\end{array}\right],\left[\begin{array}{c|c|c|c}
   B & O & \cdots & O 
\end{array}\right]\right)$, where $O$ is the zero matrix,
    \item $f_2(\mathcal{A},\mathcal{B})=(\mathcal{A}\times_{3}M^{+},\mathcal{B}\times_{3}M^{+})$,
    \item $g_1(\mathcal{A})=\mathcal{A} \times_3 M$, and
    \item $g_2\left(\left[\begin{array}{c|c|c|c}
   A^{(1)} & A^{(2)} & \cdots & A^{(q)} 
\end{array}\right]\right)=A^{(1)}$.
\end{itemize}
This diagram commutes, since
\begin{equation*}
\begin{aligned}
    &(g_2 \circ g_1 \circ \operatorname{TM}_{\langle n \rangle,p,q} \circ f_2 \circ f_1)(A,B)\\
    &=(g_2 \circ g_1 \circ \operatorname{TM}_{\langle n \rangle,p,q} \circ f_2)\left(\left[\begin{array}{c|c|c|c}
   A & O & \cdots & O 
\end{array}\right],~\left[\begin{array}{c|c|c|c}
   B & O & \cdots & O 
\end{array}\right]\right)\\
    &=(g_2 \circ g_1 \circ \operatorname{TM}_{\langle n \rangle,p,q})\left(\left[\begin{array}{c|c|c|c}
   A & O & \cdots & O 
\end{array}\right]\times_{3}M^{+},~\left[\begin{array}{c|c|c|c}
   B & O & \cdots & O 
\end{array}\right]\times_{3}M^{+}\right)\\
    &=(g_2 \circ g_1)\left(\left[\begin{array}{c|c|c|c}
   AB & O & \cdots & O 
\end{array}\right] \times_3 M^+\right)~ (\because MM^+=I_q)\\
    &=g_2\big(\left[~AB~|~O~|~\cdots~|~O~\right]\big)\\
    &=AB\\
    &=\operatorname{M}_{\langle n \rangle} (A,B).
\end{aligned}
\end{equation*}
Additionally, $f_2 \circ f_1$ is injective (possibly invertible) and $g_2 \circ g_1$ is surjective. By Lemma~\ref{commute}, we have that $\mathbf{R}(\operatorname{M}_{\langle n \rangle}) \leq \mathbf{R}(\operatorname{TM}_{\langle n \rangle,p,q})$. Therefore, 
\begin{equation*}
\omega(\operatorname{M}_{\langle n \rangle})\leq\omega(\operatorname{TM}_{\langle n \rangle,p,q}).
\end{equation*}
\end{proof}

\begin{remark}\label{rmk:InjAsymptoticEquality}
Example \ref{ex:ZeroTMTensor} shows that $\omega(\operatorname{TM}_{\langle n \rangle,p,q})<\omega(M_{\langle n \rangle})$ for some injective $M\in\mathbb{K}^{q\times p}$. In particular, the example even shows that $\omega(\operatorname{TM}_{\langle n \rangle,p,q})$ can be zero. Here, we discuss how many injective matrices satisfy the equality (\ref{eq:AsymptoticExponentRank}).

Let $M=[m_{ij}]\in\mathbb{K}^{q\times p}$ be an injective matrix (i.e., $p<q$), and let its pseudoinverse be $M^{+}=[n_{ij}]\in\mathbb{K}^{p\times q}$. Consider the maps $f_1,f_2,g_1$ and $g_2$ in the proof of Theorem \ref{thm:AsymptoticExponentRank}. Since $MM^+ \neq I_q$, then the map $f_2$ will not be helpful, so we will only use $f_1$. For $A,B \in \mathbb{K}^{n \times n}$,
\begin{equation*}
    \begin{aligned}
        &(\operatorname{TM}_{\langle n \rangle,p,q} \circ f_1)(A,B)\\
        &=\operatorname{TM}_{\langle n \rangle,p,q}\left(\left[\begin{array}{c|c|c|c}
   A & O & \cdots & O 
\end{array}\right],\left[\begin{array}{c|c|c|c}
   B & O & \cdots & O 
\end{array}\right]\right)\\
        &=\left(\left[\begin{array}{c|c|c|c}
   m_{11}A & m_{21}A & \cdots & m_{q1}A 
\end{array}\right]\triangle \left[\begin{array}{c|c|c|c}
   m_{11}B & m_{21}B & \cdots & m_{q1}B 
\end{array}\right]\right)\times_3 M^+\\
        &= \left(\left[\begin{array}{c|c|c|c}
   m_{11}^2AB & m_{21}^2AB & \cdots & m_{q1}^2AB
\end{array}\right]\right)\times_3 M^+\\
        &=\left[\begin{array}{c|c|c|c}
  \sum_{j=1}^q n_{1j}m_{j1}^2AB & \sum_{j=1}^q n_{2j}m_{j1}^2AB & \cdots & \sum_{j=1}^q n_{pj}m_{j1}^2AB
\end{array}\right].
    \end{aligned}
\end{equation*}
If $\sum_{j=1}^q n_{ij}m_{j1}^2 \neq 0$ for some $i\in [q]$, it is possible to prove that $\omega(\operatorname{TM}_{\langle n \rangle,p,q})=\omega(\operatorname{M}_{\langle n \rangle})$ by taking the surjective map 
\begin{equation*}
g_2:\mathbb{K}^{n \times n \times q} \rightarrow \mathbb{K}^{n \times n}, \left[\begin{array}{c|c|c|c}
   A^{(1)} & A^{(2)} & \cdots & A^{(q)} 
\end{array}\right] \mapsto A^{(i)}.
\end{equation*}
Therefore, the equality holds for a generic injective $M$ in the quasi-affine algebraic set of full-rank matrices in $\mathbb{K}^{q \times p}$. Note that Example \ref{ex:ZeroTMTensor} gives us the case in which we cannot choose such $i$.
\end{remark}

We have studied the asymptotic exponent of $\operatorname{TM}_{\langle n \rangle,p,q}$ on the tensor rank and now we study its total complexity analogue. The following theorem will make Theorem \ref{thm:AsymptoticExponentRank} more meaningful, in the sense that it implies that it is enough to study the asymptotic exponent on the tensor rank. Its proof is similar to the proof of $\widehat{\omega}(\operatorname{M}_{\langle n \rangle}) = \omega(\operatorname{M}_{\langle n \rangle})$ in \cite{blaeser13}.

\begin{theorem}\label{thm:AsymptoticExponentTotal}
If $M \in \mathbb{K}^{q \times p}$ is either invertible or surjective (i.e., $p \geq q$), then
\begin{equation}
\widehat{\omega}(\operatorname{TM}_{\langle n \rangle,p,q})=\omega(\operatorname{TM}_{\langle n \rangle,p,q}).
\end{equation}
\end{theorem}

\begin{proof}
$\widehat{\omega}(\operatorname{TM}_{\langle n \rangle,p,q})\geq \omega(\operatorname{TM}_{\langle n \rangle,p,q})$ is obvious from the definition of the tensor rank $\mathbf{R}(\operatorname{TM})$, which counts only the number of bilinear multiplications.

Now, we prove $\widehat{\omega}(\operatorname{TM}_{\langle n \rangle,p,q}) \leq \omega(\operatorname{TM}_{\langle n \rangle,p,q})$.  From now on, we simply write $\omega:=\omega(\operatorname{TM}_{\langle n \rangle,p,q})$. From its definition, there exists $\alpha \in \mathbb{R}$ such that for all $\epsilon >0$ there exists $ m_0>1$ such that $ m\geq m_0$ implies
\begin{equation}\label{eq:UpperBoundOfR}
    \mathbf{R}(\operatorname{TM}_{\langle n \rangle,p,q})\leq \alpha\cdot m^{\omega+\epsilon}.
\end{equation}

Let $\epsilon>0$ be given, choose an $m$ large enough and let $r:=\mathbf{R}(\operatorname{TM}_{\langle m \rangle,p})$. Now, we choose the following algorithm to multiply two $(m^i\times m^i\times p)$-sized tensors: we decompose them into $(m^{i-1}\times m^{i-1}\times p)$-sized tensors (the number of these tensors is $m\times m$) and then apply the following recursion
\begin{equation}\label{eq:recursion}
\operatorname{A}(i)\leq r\cdot \operatorname{A}(i-1)+c\cdot m^{2(i-1)},
\end{equation}
where $\operatorname{A}(i)$ denotes the number of arithmetic operations for the multiplication of two $(m^i\times m^i\times p)$-sized tensors with the chosen algorithm and where $c$ is the number of additions and scalar multiplications that are performed by the chosen bilinear algorithm.

From \eqref{eq:recursion}, we get the inequality
\begin{equation*}
\operatorname{A}(i)\leq r^i\cdot\operatorname{A}(0)+c\cdot m^{2(i-1)}\left(\sum\limits_{j=0}^{i-1}\frac{r^j}{m^{2j}} \right).
\end{equation*}
From Remark \ref{rmk:LowerBoundFromConcise}, we have that $r\geq m^2q\geq m^2$ (we assumed $q\geq 1$ at Section \ref{Sec:injective}).\\
Case 1) If $r=m^2$, then this expression becomes
\begin{equation*}
\begin{aligned}
\operatorname{A}(i)&\leq r^i\cdot\operatorname{A}(0)+c\cdot m^{2(i-1)}\left(\sum\limits_{j=0}^{i-1}1 \right)\\
&=r^i\cdot\operatorname{A}(0)+c\cdot m^{2(i-1)}i\\
&=r^i\cdot\operatorname{A}(0)+c\cdot r^{(i-1)}i.
\end{aligned}
\end{equation*}
Case 2) If $r>m^2$, then we have that
\begin{equation*}
\begin{aligned}
\operatorname{A}(i)&\leq r^i\cdot\operatorname{A}(0)+c\cdot m^{2(i-1)}\left(\sum\limits_{j=0}^{i-1}\left(\frac{r}{m^{2}}\right)^j \right)\\
&=r^i\cdot\operatorname{A}(0)+c\cdot m^{2(i-1)}\frac{m^2}{r-m^2}\left( \frac{r^i}{m^{2i}}-1\right)\\
&=r^i\cdot\operatorname{A}(0)+\frac{c}{r-m^2}\left(m^{2i} \frac{r^i}{m^{2i}}-m^{2i}\right)\\
&=r^i\cdot\operatorname{A}(0)+\frac{c}{r-m^2}\left(r^{i}-m^{2i}\right)\\
&=\left(\operatorname{A}(0)+\frac{c}{r-m^2}\right)r^i -\frac{c}{r-m^2}\cdot m^{2i}.
\end{aligned}
\end{equation*}
Thus, we have $A(i)=O(r^i)$ for all $i>0$.

Note that $\mathbf{L}(\operatorname{TM}_{\langle n\rangle,p,q}) \leq \mathbf{L}(\operatorname{TM}_{\langle n'\rangle,p,q})$ if $n \leq n'$, and the inequality $n\leq m^{\lceil \log_m n \rceil}$ holds. Hence, we obtain that
\begin{equation*}
\begin{aligned}
\mathbf{L}(\operatorname{TM}_{\langle n\rangle,p,q})&\leq \mathbf{L}(\operatorname{TM}_{\langle m^{\lceil \log_m n \rceil}\rangle,p,q}) \\
&\leq \operatorname{A}(\lceil \log_m n \rceil)\\
&= O(r^{\lceil \log_m n \rceil})\\
&= O(r^{\log_m n})\\
&= O(n^{\log_m r}).
\end{aligned}
\end{equation*}
Since $r\leq\alpha\cdot m^{\omega+\epsilon}$ (see \ref{eq:UpperBoundOfR}), then $\log_m r\leq\omega+\epsilon+\log_m \alpha$. For $\epsilon':=\epsilon+\log_m \alpha$, we have that
\begin{equation*}
\mathbf{L}(\operatorname{TM}_{\langle n\rangle,p,q})
= O(n^{\omega+\epsilon'}).
\end{equation*}
That is, $\widehat{\omega}(\operatorname{TM}_{\langle n \rangle,p,q})\leq \omega(\operatorname{TM}_{\langle n \rangle,p,q})+\delta$ for all $\delta>0$, and thus,
\begin{equation*}
    \widehat{\omega}(\operatorname{TM}_{\langle n \rangle,p,q})\leq \omega(\operatorname{TM}_{\langle n \rangle,p,q}).
\end{equation*}
\end{proof}

\subsection{Asymptotic exponent of the tensor polynomial evaluation problem}

We investigate the asymptotic exponent on the total complexity of the tensor polynomial evaluation problem under the tensor-tensor multiplication $*_M$, where $M$ is either invertible or surjective. We will not consider injective $M$, since associativity fails in this case, as well as because of Example \ref{ex:ZeroTMTensor}. Before dealing with the tensor polynomial evaluation problem, we first review the result in \cite{MR4144013} on the matrix polynomial evaluation problem, and then introduce its proof. Then, we show that the asymptotic exponent on the total complexity of the tensor polynomial evaluation problem is the same as the asymptotic exponent of the matrix multiplication.

Let $f_n(X_1,\dots,X_k)$ be a matrix polynomial  with $k$ independent noncommuting $n \times n$ matrix indeterminates $X_1,\dots,X_k$ (for instance, $f_n(X_1,X_2,X_3)=X_1X_2X_1+X_3^2+3I_n$ is an example of such a polynomial when $k=3$). The problem of evaluating any matrices $A_1,\dots,A_k\in\mathbb{K}^{n\times n}$ to $f_n$, denoted by $\operatorname{ev}_{f_n}$, is called the matrix polynomial evaluation problem. In \cite[Theorem 5.1]{MR4144013} it is asserted that, for any such $f_n$, 
\begin{equation}\label{eq:ExponentPolyMatrixEval}
    \widehat{\omega}(\operatorname{ev}_{f_n})=\widehat{\omega}(\operatorname{M}_{\langle n \rangle}).
\end{equation}
The proof of this theorem is not included in the manuscript and, thus, we provide a proof for any matrix polynomial that is not linear. Then, the proof for the $\operatorname{TM}_{\langle n \rangle, p,q}$ case can be constructed in an analogous manner. We first introduce a necessary lemma for the proof of \cite[Theorem 5.1]{MR4144013}.

\begin{lemma}\label{lemma:matrix_eval}
Let $A,B$ be two matrices in $\mathbb{C}^{n\times n}$. For any $d>1$, the matrix multiplication $AB$ can be calculated by an evaluation on the matrix power function $f_n(X)=X^d$. 
\end{lemma}
\begin{proof}
Since the following two equalities hold
\begin{equation*}
    \begin{bmatrix}
        O & A\\
        B & O
    \end{bmatrix}^2=\begin{bmatrix}
    AB & O \\
    O & BA
\end{bmatrix},~\begin{bmatrix}
        O & A & O\\
        O & O & I\\
        B & O & O
    \end{bmatrix}^3=\begin{bmatrix}
    AB & O & O\\
    O & BA & O\\
    O & O & BA
\end{bmatrix},
\end{equation*}
and, similarly, since for the block $d \times d$ ($d \geq 3$) matrix
\begin{equation}\label{eq:BlockMatrixPower}
P=\begin{bmatrix}
    O & A & O & O &\cdots & O\\
    O & O & I & O & \ddots  & \vdots \\
    \vdots & \ddots & \ddots & \ddots & \ddots & \vdots \\
    \vdots  & \ddots & \ddots & O &  I & O\\
    O & \cdots  & \cdots & \cdots &  O & I\\
    B & O & O & \cdots & O & O
\end{bmatrix},
\end{equation}
we have
\begin{equation*}
    P^d=\begin{bmatrix}
    AB &  &  &\\
     & BA &   &\\
     &  & \ddots & \\
     &   &  & BA 
\end{bmatrix},
\end{equation*}
then the matrix multiplication $AB$ can be calculated by these evaluations on the matrix power function $f_n(X)=X^d$ for all $d \geq 1$.
\end{proof}

\begin{remark}\label{rmk:matrix_eval_one-one}
Note that the $(1,1)$-block of $\begin{bmatrix}
        O & A\\
        B & O
    \end{bmatrix}$ 
is the zero matrix. Furthermore, for the matrix $P$ at (\ref{eq:BlockMatrixPower}), the top left corner block of its $s$-th power $P^s$ is the zero matrix for all $1\leq s<d$.
\end{remark}

\begin{proof}[Proof of {\cite[Theorem 5.1]{MR4144013}}]

Let $f_n$ be a nonlinear matrix polynomial of independent $n \times n$ matrix indeterminates $X_1,\dots,X_k$ with degree $d$. We show the equality (\ref{eq:ExponentPolyMatrixEval}) for $f$. It is obvious that we can evaluate $n \times n$ matrices $A_1,\dots,A_k$ on $f$ by performing the $n \times n$ matrix multiplication $\operatorname{M}_{\langle n \rangle}$ several times. Thus, 
\begin{equation*}
    \widehat{\omega}(\operatorname{ev}_{f_n})\leq\widehat{\omega}(\operatorname{M}_{\langle n \rangle}).
\end{equation*}

Now we will prove the other inequality. Let $A,B$ be $n\times n$ matrices. Consider the univariate matrix polynomial $f_{m}(X_1,X_1,\dots,X_1)$, where $X_1$ denotes an $m \times m$ matrix indeterminate. Let the largest exponent of $X$ in $f_{m}(X_1,X_1,\dots,X_1)$ be $d$, and let $f_m(X_1,X_1,\dots,X_1)$ have the constant $m \times m$ matrix term $C_m$. Take the $P$ at (\ref{eq:BlockMatrixPower}), then Lemma \ref{lemma:matrix_eval} and Remark \ref{rmk:matrix_eval_one-one} imply that the the top left corner block of $f_{dn}(P^d)$ is 
\begin{equation*}
    AB+C_{dn}(1:n,1:n),
\end{equation*}
where $C_{dn}(1:n,1:n)$ denotes the top left corner block of $C_{dn}$. Therefore, $\operatorname{M}_{\langle n\rangle}(A,B)$ can be calculated by evaluating $P$ to $f_{dn}$ for fixed $d$ and then subtracting the fixed constant matrix $C_{dn}(1:n,1:n)$, and so
\begin{equation*}
    \widehat{\omega}(\operatorname{ev}_{f_n})\geq\widehat{\omega}(\operatorname{M}_{\langle n \rangle}).
\end{equation*}
\end{proof}

\begin{remark}
If $f_n$ is a linear matrix polynomial, then no matrix multiplication has been computed when evaluating matrices, and thus it is hard to compare the total complexity of matrix multiplication and matrix polynomial evaluation. If the equality $\omega(\operatorname{M}_{\langle n\rangle})=2$ holds, then (\ref{eq:ExponentPolyMatrixEval}) also holds. However, the present best bounds for $\omega(\operatorname{M}_{\langle n\rangle})$ are $2 \leq \omega(\operatorname{M}_{\langle n\rangle}) \leq 2.371339$ \cite{MR4863478}. In the case of multiplicative complexity (see \cite[Section 4.1]{MR1440179}), the asymptotic exponent of the linear matrix polynomial evaluation is even zero.
\end{remark}
 
\begin{theorem}\label{Thm:AsymptoticPolynomial}
Let $M \in \mathbb{K}^{q \times p}$ be either invertible or surjective (i.e., $p \geq q$) and let $f_n(\mathcal{X}_1,\dots,\mathcal{X}_k)$ be a tensor polynomial under the $*_M$-product, where $\mathcal{X}_1,\dots,\mathcal{X}_k$ are $k$ independent $n \times n$ matrix indeterminates. Let $\mathrm{ev}_{f_n}$ be the tensor polynomial evaluation problem. Then,
\begin{equation}
\widehat{\omega}(\operatorname{ev}_{f_n}) = \widehat{\omega}(\operatorname{TM}_{\langle n \rangle,p})=\omega(\operatorname{M}_{\langle n \rangle}).
\end{equation}
\end{theorem}
\begin{proof}
    Since the $*_M$-product can be performed block-wisely, then the first equality is proved similarly as in the matrix polynomial evaluation problem above. In addition, the second equality holds by Theorems \ref{thm:AsymptoticExponentTotal} and \ref{thm:AsymptoticExponentRank}.
\end{proof}

\section{Non-algebraic tensor functions}\label{Sec:non-algebraic}

\subsection{Tensor means under the $*_M$-product}\label{Sec:means}
For two arbitrary third-order tensors $\mathcal{A},\mathcal{B}\in\mathbb{K}^{n\times n\times p}$, its \emph{arithmetic mean} is
\begin{equation*}
    \frac{\mathcal{A}+\mathcal{B}}{2}.
\end{equation*}
This definition has been easily generalized from the matrix arithmetic mean. However, in order to generalize some other matrix means, we need the multiplication operation for tensors. 

Let $A,B$ be $n \times n$ Hermitian positive-definite matrices. Then the \emph{geometric mean of $A$ and $B$}, denoted by $A\#B$, is defined as 
\begin{equation*}
    A \# B=A^{\frac{1}{2}}(A^{-\frac{1}{2}}BA^{-\frac{1}{2}})^{\frac{1}{2}}A^{\frac{1}{2}}.
\end{equation*}
Note that it is defined by using matrix multiplication, inverse, and square root. In addition, the \emph{Wasserstein mean of $A$ and $B$}, denoted by $A \diamond B$, is defined as
\begin{equation*}
    A \diamond B=\frac{1}{4}(A+B+(AB)^{\frac{1}{2}}+(BA)^{\frac{1}{2}}).
\end{equation*}
It also requires matrix multiplication and square root. In fact, both the geometric mean and the Wasserstein mean can be defined as the midpoint of the geodesic connecting $A$ and $B$ on Riemannian manifolds of $n \times n$ Hermitian positive-definite matrices \cite{MR2198952,MR3992484}.

In \cite{ju2024}, the authors introduced the notion of geometric mean for T-positive-definite tensors. The goal of this section is to generalize this concept by defining the geometric mean of two tensors under the family of $*_M$-products, and furthermore, to introduce the concept of Wasserstein mean for two tensors under the $*_M$-product. In Example~\ref{ex:CounterexampleResultant} and \ref{ex:hyper} we show how the determinantal property, which is satisfied for the matrix geometric mean, is not true for the tensor case. 

\subsubsection{Positive definiteness and pseudo-positive-definiteness}\label{Sec:PD_PSEUDOPD}

In order to define the desired tensor means, first we need to introduce some necessary concepts for their construction, such as the positive definiteness for third-order tensors with respect to the $*_M$-product. The authors of \cite{MR3394164} equipped $\mathbb{K}^{m\times 1\times p}$, the space of lateral slices, with an inner product under which they were able to define a norm. 

\begin{definition}
Let $M\in\mathbb{K}^{q\times p}$ be of full rank, and let $\mathcal{A},\mathcal{B}\in\mathbb{K}^{m\times 1\times p}$. The \emph{$*_M$-inner product of $\mathcal{A}$ and $\mathcal{B}$}, denoted by $\langle \mathcal{A},\mathcal{B} \rangle_{*_M}$, is defined as
   \begin{equation*}
       \langle \mathcal{A},\mathcal{B} \rangle_{*_M}= \mathcal{B}^H*_M\mathcal{A},
   \end{equation*}
which lives in $\mathbb{K}^{1\times 1\times p}$.
\end{definition}

Given the presence of this inner product, it is natural to attempt a generalization of the concept of matrix positive definiteness to that of tensors under the $*_M$-product. We provide this generalized definition, as well as the problems that this generalization creates, and finally a second definition on which the theory of tensor means will be constructed. This second attempt will be defined as the slice-wise positive definiteness in the image of $\times_3 M$ and, when $M$ is the discrete Fourier transformation, it coincides with the definition of T- positive definiteness under the inner product introduced in \cite{MR4198732}. See Remark \ref{rem:implications}.

The first attempt at the generalization of the definition of positive definiteness is as follows.

\begin{definition}\label{def:PD}
Let $M\in\mathbb{K}^{q\times p}$ be of full rank. Then $\mathcal{A}\in\mathbb{K}^{n\times n \times p}$ is said to be \emph{$*_M$-positive-definite} if
\begin{itemize}
    \item [(\romannumeral1)] $\mathcal{A}$ is $*_M$-Hermitian,
    \item [(\romannumeral2)] $\langle \mathcal{X},\ \mathcal{A}*_M \mathcal{X}\rangle_{*_{M}}> 0$ for all $\mathcal{X}\in\mathbb{K}^{n\times 1 \times p}\setminus \{\mathcal{O}\}$,
\end{itemize}
where $\mathcal{O}$ denotes the zero tensor, and the inequality on $\mathbb{K}^{1\times 1 \times p}$ is defined as
\begin{equation*}
    \left[\begin{array}{c|c|c}
     y_1 & \cdots & y_p
\end{array}\right]>0 \Leftrightarrow y_i\geq0~\text{for all}~i \in [p]~\text{and there is}~j \in [p]~\text{such that}~y_j>0.
\end{equation*}
\end{definition}

\begin{remark}
    The definition of the inequality on $\mathbb{K}^{1 \times 1 \times p}$ is reasonable, because it is possible for no $*_M$-positive-definite tensor to exist if we define the inequality as 
    \begin{equation}\label{eq:WrongInequalityCondition}
    \left[\begin{array}{c|c|c}
     y_1 & \cdots & y_p
\end{array}\right]>0 \Leftrightarrow y_i>0~\text{for all}~i \in [p].
    \end{equation}
    Even for the case when $M=I_2$, the tensor $\mathcal{A}=\left[~\mathcal{A}^{(1)}~\middle|~\mathcal{A}^{(2)}~\right] \in \mathbb{K}^{2 \times 2 \times 2}$, with
    \begin{equation*}
        \mathcal{A}^{(1)}=\begin{bmatrix}
            2 & -1\\
            -1 & 2
        \end{bmatrix},~ \mathcal{A}^{(2)}=\begin{bmatrix}
            2 & 0\\
            0 & 2
        \end{bmatrix},
    \end{equation*}
    does not satisfy the condition in (\ref{eq:WrongInequalityCondition}), since 
    \begin{equation*}
        \mathcal{X}=\left[
        \begin{array}{c|c}
             1 & 0 \\
             1 & 0
        \end{array}
        \right]
    \end{equation*}
    satisfies $\langle \mathcal{X},\mathcal{A}*_M \mathcal{X}\rangle_{*_M}=[~2~|~0~]$. That is, the condition in (\ref{eq:WrongInequalityCondition}) is difficult to satisfy even when $M=I_n$ and $\mathcal{A}$ consists of positive-definite frontal slices. 
\end{remark}

We will not adopt this definition when developing a theory of tensor means, because of the non-existence and non-$*_M$-positive-definiteness of the square root of $*_M$-positive-definite tensors. 

\begin{definition}
    Let $M\in\mathbb{K}^{q\times p}$ be of full rank, and let $\mathcal{A},\mathcal{B}\in\mathbb{K}^{n\times n \times p}$. $\mathcal{B}$ is called a \emph{square root of $\mathcal{A}$} if $\mathcal{B}*_M\mathcal{B}=\mathcal{A}$. 
\end{definition} 

\begin{example}\label{ex:NoRealSquareRootPositiveDefinite}
    Let $M=\begin{bmatrix}
        0 & 1\\
        -1 & 1
    \end{bmatrix}$ (invertible) so that $M^{-1}=\begin{bmatrix}
        1 & -1\\
        1 & 0
    \end{bmatrix}$.  Take the real tensor
    \begin{equation*}
        \mathcal{A}=\left[\begin{array}{cc|cc}
             2& 0& 1 & 0  \\
             0 & 2& 0 & 1
        \end{array}
        \right].
    \end{equation*}
    It is easy to check that $\mathcal{A}$ is $*_M$-positive-definite. Now, let us calculate its square root. Note that
    \begin{equation*}
        \mathcal{A} \times_3 M=\left[\begin{array}{cc|cc}
             1& 0& -1 & 0  \\
             0 & 1& 0 & -1
        \end{array}
        \right].
    \end{equation*} 
    If $\mathcal{B} \in \mathbb{K}^{2 \times 2 \times 2}$ is a square root of $\mathcal{A}$, then $\mathcal{B} \times_3 M$ must be a square root of $\mathcal{A} \times_3 M$ under the face-wise product. However, the second frontal slice of $\mathcal{A} \times_3 M$ has the unique square root
    \begin{equation*}
    \begin{bmatrix}
        i & 0\\
        0 & i
    \end{bmatrix}
    \end{equation*}
    which is neither real nor Hermitian. That is, $\mathcal{A}$ is a $*_M$-positive-definite real tensor without a real and Hermitian (symmetric) square root. 
\end{example}

Note that every positive-definite matrix has a unique Hermitian positive-definite square root, and this property is widely used in the theory of matrix means. This is why we will not adopt the $*_M$-positive-definiteness. Instead, we will use the following notion.

\begin{definition}
    Let $M\in\mathbb{K}^{q\times p}$ be of full rank. A tensor $\mathcal{A}\in\mathbb{K}^{n\times n\times p}$ is said to be \emph{$*_M$-pseudo-positive-definite} if the frontal slices $(\mathcal{A}\times_3 M)^{(i)}$ are Hermitian positive-definite for all $i\in [q]$.
\end{definition}

In the case of invertible or surjective maps $M$, a $*_M$-pseudo-positive-definite tensor always exists, since $\times_3 M:\mathbb{K}^p \rightarrow \mathbb{K}^q$ is either invertible or surjective and so $\times_3 M:\mathbb{K}^{n \times n \times p} \rightarrow \mathbb{K}^{n \times n \times q}$ is also either invertible or surjective. However, for injective $M$, the existence of a $*_M$-pseudo-positive-definite tensor is not guaranteed.

\begin{example}\label{ex:InjectiveNoPD}
     Let $M=\begin{bmatrix} 1 \\ -1\end{bmatrix}$. Then, for every  $A\in\mathbb{K}^{n\times n\times 1}$ (corresponding to an $n \times n$ matrix),
     \begin{equation*}
         A \times_3 M=\left[~A~\middle|~-A~\right].
     \end{equation*}
     When $A$ is positive-definite, $-A$ must be negative-definite. Hence, $M$ does not admit a $*_M$-pseudo-positive-definite tensor in this case.
\end{example}

There is no identity tensor in $\mathbb{K}^{n \times n \times p}$ in the case of the injective matrix $M$ in Example \ref{ex:InjectiveNoPD}. However, if $M$ admits $\left[\begin{array}{c|c|c|c}
     1 & 1 & \cdots & 1
\end{array}\right]$ in the image of $\times_3 M:\mathbb{K}^{1 \times 1 \times p} \rightarrow \mathbb{K}^{1 \times 1 \times q}$, then an identity tensor exists (see Remark \ref{rmk:RemarkIdentityTensor}), which is $*_M$-pseudo-positive-definite by definition, and so a $*_M$-pseudo-positive-definite tensor exists. 

\begin{lemma}\label{lemma:invertible}
Let $M\in\mathbb{K}^{q\times p}$ be either invertible or surjective. If $\mathcal{A} \in\mathbb{K}^{n\times n \times p}$ is $*_M$-pseudo-positive-definite, then $\mathcal{A}$ is invertible. 
\end{lemma}
\begin{proof}
If $\mathcal{A}\in\mathbb{K}^{n\times n \times p}$ is $*_M$-pseudo-positive-definite, then all slices of $\mathcal{A} \times_3 M$ are positive-definite and so invertible. In addition, the invertibility or surjectivity of $M$ guarantees the existence of the stack of inverses in the image of $\times_3 M$ in $\mathbb{K}^{n \times n \times q}$.
\end{proof}

If $M$ is injective, then this lemma does not hold in general, even if the identity tensor exists.

\begin{example}\label{ex:InjectivePositiveDefiniteNonInvertible}
    Let $M=\begin{bmatrix}
        1 & 0 \\
        0 & 1\\
        1 & 1
    \end{bmatrix}$. Then 
    \begin{equation*}
        \mathcal{A}=\left[\begin{array}{cc|cc}
            1 & 0 & 1 & 0  \\
            0 & 1 & 0 & 1
        \end{array}\right]
    \end{equation*}
    is $*_M$-pseudo-positive-definite, because
    \begin{equation*}
        \mathcal{A}\times_3 M=\left[\begin{array}{cc|cc|cc}
            1 & 0 & 1 & 0 & 2 & 0  \\
            0 & 1 & 0 & 1 & 0 &2
        \end{array}\right].
    \end{equation*}
    If $\mathcal{A}$ had an inverse $\mathcal{A}^{-1}$, then 
    \begin{equation*}
        \mathcal{A}^{-1}\times_3 M=\left[\begin{array}{cc|cc|cc}
            1 & 0 & 1 & 0 & \frac{1}{2} & 0  \\
            0 & 1 & 0 & 1 & 0 &\frac{1}{2}
        \end{array}\right].
    \end{equation*}
    However, if the first two frontal slices of $\mathcal{A}\times_3 M$ are $I_2$, then the third frontal slice must be $2I_2$, and not $\frac{1}{2}I_2$. Therefore, $\mathcal{A}$ does not have an inverse. 
\end{example}

Because of Examples \ref{ex:InjectiveNoPD} and \ref{ex:InjectivePositiveDefiniteNonInvertible}, we will not deal with the injective case for the rest of this section.

\begin{lemma}\label{lemma:sqrt}
Let $M\in\mathbb{K}^{q\times p}$ be either invertible or surjective. If $\mathcal{A}\in\mathbb{K}^{n\times n \times p}$ is $*_M$-pseudo-positive-definite, then there exists a $*_M$-pseudo-positive-definite $k$-th root of $\mathcal{A}$ for every positive integer $k$. In particular, if $M$ is invertible, then such $k$-th root is uniquely determined. 
\end{lemma}
\begin{proof}
If $\mathcal{A}\in\mathbb{K}^{n\times n \times p}$ is $*_M$-pseudo-positive-definite, then each slice of $\mathcal{A} \times_3 M$ is positive-definite and so has a unique positive-definite $k$-th root. In addition, the invertibility or surjectivity of $M$ guarantees the existence of the stack of $k$-th roots in the image of $\times_3 M$ in $\mathbb{K}^{n \times n \times q}$. Hence, the proof for the existence  is done. Furthermore, if $M$ is invertible, then both $\times_3 M,~ \times_3M^{-1}:\mathbb{K}^{n \times n \times p} \rightarrow \mathbb{K}^{n \times n \times p}$ are injective (possibly invertible, and in fact they are bijective), and so such $k$-th root is uniquely determined.
\end{proof}

If $M$ is surjective, then there can be several $*_M$-pseudo-positive-definite $k$-th roots of a $*_M$-pseudo-positive-definite tensor. See the following example.
\begin{example}\label{ex:SurjectiveSquareRootNotUnique}
    Let $M=\begin{bmatrix}
        1 & 0 & 1\\
        0 & 1 & 0
    \end{bmatrix}$ so that $M^+=\begin{bmatrix}
        1 & 0 \\ 0 & 1\\ 0 & 0
    \end{bmatrix}$. Consider 
    \begin{equation*}
        \mathcal{A}=\left[\begin{array}{cc|cc|cc}
             2 & 0 & 4 & 0 & 2 & 0  \\
             0 & 2 & 0 & 4 & 0 & 2
        \end{array}\right].
    \end{equation*}
    It is $*_M$-pseudo-positive-definite because each frontal slice of
    \begin{equation*}
        \mathcal{A} \times_3 M=\left[\begin{array}{cc|cc}
             4 & 0 & 4 & 0   \\
             0 & 4 & 0 & 4 
        \end{array}\right]
    \end{equation*}
    is positive-definite. Both
    \begin{equation*}
        \mathcal{B}=\left[\begin{array}{cc|cc|cc}
             1 & 0 & 2 & 0 & 1 & 0  \\
             0 & 1 & 0 & 2 & 0 & 1
        \end{array}\right]~\text{and}~\mathcal{B}'=\left[\begin{array}{cc|cc|cc}
             3 & 0 & 2 & 0 & -1 & 0  \\
             0 & 3 & 0 & 2 & 0 & -1
        \end{array}\right]
    \end{equation*} satisfy 
    \begin{equation*}
        \mathcal{B} \times_3 M=\mathcal{B}' \times_3 M=\left[\begin{array}{cc|cc}
             2 & 0 & 2 & 0   \\
             0 & 2 & 0 & 2
        \end{array}\right],
    \end{equation*}
    and so they are square roots of $\mathcal{A}$, but $\mathcal{B} \neq \mathcal{B}'$.
\end{example}

\begin{remark}\label{rmk:AssumptionForUniqueSquareRoot}
    Because of Examples \ref{ex:InjectiveNoPD}, \ref{ex:InjectivePositiveDefiniteNonInvertible}, and  \ref{ex:SurjectiveSquareRootNotUnique}, we will assume that $M$ is invertible when we deal with $*_M$-pseudo-positive-definiteness. In this case, we let $\mathcal{A}^{\frac{1}{2}}$ denote the uniquely determined square root of $\mathcal{A}$ (from Lemma \ref{lemma:sqrt}).
\end{remark}

We lastly discuss implications between $*_M$-positive-definiteness and $*_M$-pseudo-positive-definiteness, and then proceed to the subsection on tensor means.

\begin{remark}\label{rem:implications}
Here, we show that there are no implications between $*_M$-positive-definiteness and $*_M$-pseudo-positive-definiteness, and give a case where they are equivalent.
\begin{itemize}
    \item [(\romannumeral1)] Example \ref{ex:NoRealSquareRootPositiveDefinite} introduced an example which is $*_M$-positive-definiteness but not $*_M$-pseudo-positive-definiteness.
    \item [(\romannumeral2)] In Example \ref{ex:NoRealSquareRootPositiveDefinite}, if we take 
    \begin{equation*}
        \mathcal{B}=\left[\begin{array}{cc|cc}
             0& 0& 1 & 0  \\
             0 & 0& 0 & 1
        \end{array}\right],
    \end{equation*}
    then
    \begin{equation*}
        \mathcal{B} \times_3 M=\left[\begin{array}{cc|cc}
             1& 0& 1 & 0  \\
             0 & 1& 0 & 1
        \end{array}\right]
    \end{equation*}
    and so it is $*_M$-pseudo-positive-definite. Indeed,
    \begin{equation}\label{eq:CounterExampleImplication}
        \mathcal{X}=\left[\begin{array}{c|c}
             1& 0  \\
             0 & 0
        \end{array}\right]
    \end{equation}
    satisfies $\langle \mathcal{X},\mathcal{B}*_M \mathcal{X} \rangle_{*_M}=[-1~|~0]$, and so $\mathcal{B}$ is not $*_M$-positive-definite.
    \item [(\romannumeral3)] In the case of the T-product, i.e., in the case where $M$ is a discrete Fourier matrix, it is proved in \cite[Theorem 4.4]{MR4198732} that the two types of positive definiteness are equivalent when replacing the inner product in Definition \ref{def:PD} by the Frobenius inner product $\langle \mathcal{A},\mathcal{B}\rangle=\sum\limits_{i=1}^{m}\sum\limits_{j=1}^n\sum\limits_{k=1}^p \overline{a}_{ijk}b_{ijk}$.
    \item [(\romannumeral4)] Even when $M$ is orthogonal, the definitions of $*_M$-positive-definiteness and $*_M$-pseudo-positive-definiteness are not equivalent. Let $M=\begin{bmatrix}
        0 & 1\\
        -1 & 0
    \end{bmatrix}$ so that $M^{-1}=\begin{bmatrix}
            0 & -1\\
            1 & 0
        \end{bmatrix}$. Take
        \begin{equation*}
            \mathcal{C}=\left[\begin{array}{cc|cc}
             -1 & 0& 1 & 0  \\
             0 & -1 & 0 & 1
        \end{array}\right]
        \end{equation*} so that
        \begin{equation*}
            \mathcal{C} \times_3 M=\left[\begin{array}{cc|cc}
             1 & 0& 1 & 0  \\
             0 & 1 & 0 & 1
        \end{array}\right].
        \end{equation*}
        Hence, it is $*_M$-pseudo-positive-definite. However, (\ref{eq:CounterExampleImplication}) satisfies $\langle \mathcal{X},\mathcal{C}*_M \mathcal{X} \rangle_{*_M}=[-1~|~0]$, and so $\mathcal{C}$ is not $*_M$-positive-definite.
\end{itemize}
\end{remark}

\subsubsection{Geometric mean for pseudo-positive definite tensors}

Throughout this subsection, the matrix $M\in\mathbb{K}^{p \times p}$ is invertible unless otherwise stated. This assumption is a consequence from Remark \ref{rmk:AssumptionForUniqueSquareRoot}. We now generalize the geometric mean for tensors with respect to the $*_M$-product.

\begin{definition}
Let $\mathcal{A},\mathcal{B}\in\mathbb{K}^{n\times n \times p}$ be $*_M$-pseudo-positive-definite. The \textit{geometric mean of $\mathcal{A}$ and $\mathcal{B}$ under the $*_M$-product}, denoted by $\mathcal{A}\#_{*_M}\mathcal{B}$, is defined as
\begin{equation*}
\mathcal{A}\#_{*_M}\mathcal{B}=\mathcal{A}^{\frac{1}{2}}*_M(\mathcal{A}^{-\frac{1}{2}}*_M\mathcal{B}*_M\mathcal{A}^{-\frac{1}{2}})^{\frac{1}{2}}*_M\mathcal{A}^{\frac{1}{2}}.
\end{equation*}
\end{definition}

\begin{remark}\label{rmk:WellDefinedGM}
    It is well-defined by Lemmas \ref{lemma:invertible}, \ref{lemma:sqrt} and Remark \ref{rmk:UniquenessOfInverse}.
\end{remark}

\begin{notation}
    For $M \in \mathbb{K}^{q \times p}$ and $\mathcal{A} \in \mathbb{K}^{m \times n \times p}$, we write
    \begin{equation*}
        \widehat{\mathcal{A}}:=\mathcal{A} \times_3 M.
    \end{equation*}
\end{notation}

\begin{remark}\label{rmk:FaceWiseGeometricMean}
    The following observation is important when performing the calculations:
    \begin{equation*}
        \begin{aligned}
            &((\mathcal{A}\#_{*_M}\mathcal{B}) \times_3 M)^{(i)} \\
            &=(\mathcal{A}^{\frac{1}{2}}*_M(\mathcal{A}^{-\frac{1}{2}}*_M\mathcal{B}*_M\mathcal{A}^{-\frac{1}{2}})^{\frac{1}{2}}*_M\mathcal{A}^{\frac{1}{2}}\\
            &=\left((\widehat{\mathcal{A}})^{(i)}\right)^{\frac{1}{2}}  \left(\left((\widehat{\mathcal{A}})^{(i)}\right)^{-\frac{1}{2}}  (\widehat{\mathcal{B}})^{(i)} \left((\widehat{\mathcal{A}})^{(i)}\right)^{-\frac{1}{2}}\right)^{\frac{1}{2}}  \left((\widehat{\mathcal{A}})^{(i)}\right)^{\frac{1}{2}}\\
            &=(\widehat{\mathcal{A}})^{(i)}\# (\widehat{\mathcal{B}})^{(i)}.
        \end{aligned}
    \end{equation*}
\end{remark}

\begin{lemma}
Let $\mathcal{A}, \mathcal{B} \in \mathbb{K}^{n\times n\times p}$ be $*_M$-pseudo-positive-definite. Then their geometric mean $\mathcal{A}\#_{*_M}\mathcal{B}$ is $*_M$-pseudo-positive-definite.
\end{lemma}
\begin{proof}
It is clear by Remark \ref{rmk:FaceWiseGeometricMean}. 
\end{proof}

\begin{theorem}\label{thm:RiccatiTensorEquation}
    Let $\mathcal{A},\mathcal{B}\in\mathbb{K}^{n\times n\times p}$ be $*_M$-pseudo-positive-definite. Then $\mathcal{A}\#_{*_M}\mathcal{B}$ is the unique $*_M$-pseudo-positive-definite solution of the Riccati tensor equation:
\begin{equation}\label{eq:RiccatiTensorEquation}
\mathcal{X}*_M\mathcal{A}^{-1}*_M\mathcal{X}=\mathcal{B}.
\end{equation}
\end{theorem}
\begin{proof}
It is well known that, for $n \times n$ Hermitian positive-definite matrices $A$ and $B$, their geometric mean $A \# B$ is the unique Hermitian positive-definite solution of the Riccati matrix equation \cite{MR1864051}:
\begin{equation*}
    XA^{-1}X=B.
\end{equation*}
Thus, the assertion is obvious by Remark \ref{rmk:FaceWiseGeometricMean}.
\end{proof}

If we perform $\times_3 M$ on the both sides of the Riccati tensor equation (\ref{eq:RiccatiTensorEquation}), then it is easy to notice that this equation is equivalent to the following system of Riccati matrix equations
\begin{equation*}
\begin{cases}
    \widehat{\mathcal{X}}^{(1)} {\left(\widehat{\mathcal{A}}^{(1)}\right)}^{-1} \widehat{\mathcal{X}}^{(1)}&=\widehat{\mathcal{B}}^{(1)},\\
    &\vdots\\
    \widehat{\mathcal{X}}^{(p)} {\left(\widehat{\mathcal{A}}^{(p)}\right)}^{-1} \widehat{\mathcal{X}}^{(p)}&=\widehat{\mathcal{B}}^{(p)}.\end{cases}
\end{equation*}
If we find the Hermitian positive-definite solutions $\widehat{X}^{(i)}$ by using numerical methods, then we can easily obtain $\mathcal{X}=\widehat{\mathcal{X}} \times_3 M^{-1}$. We refer \cite{MR2896454} for numerical methods to find solutions of Riccati matrix equations.

Recall that the space $H^{n \times n}_{++}$ of $n \times n$ Hermitian positive-definite matrices is a smooth manifold, and its tangent space at a point $P \in H^{n \times n}_{++}$ is the space $H^{n \times n}$ of $n \times n$ Hermitian matrices. Additionally, the inner product $\hat{g}_P$, which is locally defined by
\begin{equation*}
  \hat{g}_P(A,B)=\operatorname{tr}(P^{-1}AP^{-1}B)  
\end{equation*}
at $P \in H^{n \times n}_{++}$, gives us a Riemannian manifold $(H^{n \times n}_{++},~\hat{g})$. On this Riemannian manifold, a geodesic connecting $P,Q \in H^{n \times n}_{++}$ is uniquely determined, and the midpoint of this geodesic is the geometric mean $A \# B$ \cite{MR2198952}. 

Now, consider the spaces $\mathbb{H}^{n \times n \times p}_{++}$ of $*_M$-pseudo-positive-definite tensors of size $n \times n \times p$ and $\mathbb{H}^{n \times n \times p}$ of $*_M$-Hermitian tensors of size $n \times n \times p$. The image of $\times_3 M:\mathbb{H}^{n \times n \times p}_{++} \rightarrow \mathbb{K}^{n \times n \times p}$ is the product space of $p$-copies of $H^{n \times n}_{++}$. Since both maps $\times_3 M^{-1}$ are obviously smooth, then $\mathbb{H}^{n \times n \times p}_{++}$ is a smooth manifold. Additionally, $\times_3 M$ is smooth and linear so its differential is 
itself, and the image of $\times_3 M: \mathbb{H}^{n \times n \times p} \rightarrow \mathbb{K}^{n \times n \times p}$ is the product space of $p$-copies of $H^{n \times n}$. Therefore, it is natural to give a Riemannian metric $g$ on $\mathbb{H}^{n \times n \times p}_{++}$, where $g$ is locally defined as
\begin{equation*}
    g_{\mathcal{P}}(\mathcal{A},\mathcal{B})=\sum_{i=1}^p \hat{g}_{\widehat{\mathcal{P}}^{(i)}}\left((\widehat{A})^{(i)},(\widehat{B})^{(i)}\right)
\end{equation*}
at $\mathcal{P} \in \mathbb{H}^{n \times n \times p}_{++}$. Then, the Riemannian manifold $(\mathbb{H}^{n \times n \times p}_{++},~g)$ is isometric to the product space $\prod_{i=1}^p (H^{n \times n}_{++},~\hat{g})$. From this construction and Remark \ref{rmk:FaceWiseGeometricMean}, we obtain the following theorem.
\begin{theorem}\label{thm:MidpointGeodesic}
    For any $\mathcal{A},\mathcal{B}\in (\mathbb{H}^{n\times n\times p}_{++},~g)$, $\mathcal{A}\#_{*_M}\mathcal{B}$ is the midpoint of the uniquely determined geodesic connecting $\mathcal{A}$ and $\mathcal{B}$.
\end{theorem}

A geodesic in a Riemannian manifold is parametrized by $t \in [0,1]$. In particular, it has been proved in \cite{MR2198952} that the geodesic from $A\in H^{n \times n}_{++}$ to $B \in H^{n \times n}_{++}$ is
\begin{equation*}
    \gamma: t \in [0,1] \mapsto A^{\frac{1}{2}}(A^{-\frac{1}{2}}BA^{-\frac{1}{2}})^{t}A^{\frac{1}{2}} \in H^{n \times n}_{++}.
\end{equation*}
From this fact, the \emph{weighted geometric mean}
\begin{equation*}
    A \#_t B:=A^{\frac{1}{2}}(A^{-\frac{1}{2}}BA^{-\frac{1}{2}})^{t}A^{\frac{1}{2}}
\end{equation*}
was defined for a fixed $t \in [0,1]$. From the construction above, we also obtain that 
\begin{equation*}
    \Gamma: t\in [0,1] \mapsto \mathcal{A}^{\frac{1}{2}}*_M(\mathcal{A}^{-\frac{1}{2}}*_M\mathcal{B}*_M\mathcal{A}^{-\frac{1}{2}})^{t}*_M\mathcal{A}^{\frac{1}{2}} \in \mathbb{H}^{n \times n \times p}_{++}
\end{equation*}
is the geodesic from $\mathcal{A} \in \mathbb{H}^{n \times n \times p}_{++}$ to $\mathcal{B} \in \mathbb{H}^{n \times n \times p}_{++}$. Therefore, for a fixed $t \in [0,1]$, it is natural to define the \emph{weighted geometric mean of $\mathcal{A},\mathcal{B} \in \mathbb{H}^{n \times n \times p}_{++}$ under the $*_M$-product} by
\begin{equation*}
    \mathcal{A} \#_{*_M,t} \mathcal{B}:=\mathcal{A}^{\frac{1}{2}}*_M(\mathcal{A}^{-\frac{1}{2}}*_M\mathcal{B}*_M\mathcal{A}^{-\frac{1}{2}})^{t}*_M\mathcal{A}^{\frac{1}{2}}.
\end{equation*}
Furthermore, the multivariate geometric mean (i.e., Karcher mean) for more than two $*_M$-pseudo-positive-definite tensors can be defined as the barycenter of those tensors in the Riemannian manifold.

The properties of the matrix geometric mean have been extensively analized in literature, such as in \cite{MR3443454, MR2944035, MR563399, MR2954480}. We condense some of the properties for the geometric mean of matrices in the following lemma.

\begin{lemma}
    Let $A,B\in H^{n\times n}_{++}$. The matrix geometric mean $A\#B$ has the following properties:
    \begin{itemize}
        \item [(\romannumeral1)] \emph{Idempotence}: $A \# A=A$.
        \item [(\romannumeral2)] \emph{Inversion order-reversing property}: $(A\#B)^{-1}=B^{-1}\#A^{-1}$.
        \item [(\romannumeral3)] \emph{Commutativity}: $A\#B=B\#A$.
        \item [(\romannumeral4)] \emph{Congruence invariance}: $(C^H A C)\#(C^H B C)=  C^H(A\#B)C$, for any invertible $C$.
        \item [(\romannumeral5)] \emph{Joint homogeneity}: for $c>0$, $(cA)\# (cB)=c(A\#B)$.
        \item [(\romannumeral6)] \emph{Consistency with scalars}: if $a>0,b>0$, $A=aI$ and $B=bI$, then $A\#B=\sqrt{ab}I$.
        \item [(\romannumeral7)] \emph{Determinantal identity}: $\operatorname{det}(A\#B)=\sqrt{\operatorname{det}A\operatorname{det}B}$.
    \end{itemize}
\end{lemma}

By this lemma and Remark \ref{rmk:FaceWiseGeometricMean}, we can easily prove the following proposition:
\begin{proposition}
    Let $\mathcal{A},\mathcal{B} \in \mathbb{H}^{n\times n\times p}_{++}$. The tensor geometric mean $\mathcal{A}\#_{*_M}\mathcal{B}$ under the $*_M$-product has the following properties:
    \begin{itemize}
        \item [(\romannumeral1)] \emph{Idempotence}: $\mathcal{A} \#_{*_M} \mathcal{A}=\mathcal{A}$.
        \item [(\romannumeral2)] \emph{Inversion order-reversing property}: $(\mathcal{A}\#_{*_M}\mathcal{B})^{-1}=\mathcal{B}^{-1}\#_{*_M}\mathcal{A}^{-1}$.
        \item [(\romannumeral3)] \emph{Commutativity}: $\mathcal{A}\#_{*_M}\mathcal{B}=\mathcal{B}\#_{*_M}\mathcal{A}$.
        \item [(\romannumeral4)] \emph{Congruence invariance}: $(\mathcal{C}^H *_M \mathcal{A} *_M \mathcal{C})\#_{*_M}(\mathcal{C}^H  *_M\mathcal{B} *_M \mathcal{C})= \mathcal{C}^H *_M (\mathcal{A} \#_{*_M}\mathcal{B}) *_M \mathcal{C}$, for any invertible $\mathcal{C}$.
        \item [(\romannumeral5)] \emph{Joint homogeneity}: for $c>0$, $(c\mathcal{A})\#_{*_M} (c\mathcal{B})=c(\mathcal{A}\#_{*_M}\mathcal{B})$.
        \item [(\romannumeral6)] \emph{Consistency with scalars}: If, for $a>0,b>0$ and the identity tensor $\mathcal{I}$, $\mathcal{A}=a\mathcal{I}$ and $\mathcal{B}=b\mathcal{I}$, then $\mathcal{A}\#_{*_M}\mathcal{B}=\sqrt{ab}\mathcal{I}$.
    \end{itemize}
\end{proposition}
 
Note that this proposition does not contain an analogue of the determinantal identity property. Now, first we must consider an analogue of the determinant for third-order tensors. One of generalizations of the determinant is the \emph{resultant}. Briefly, the resultant of homogeneous polynomials $F_0,F_1,...,F_n \in \mathbb{C}[x_0,...,x_n]$ is an object presenting whether the polynomials have a nontrivial common zero in $\mathbb{C}^{n+1}$ or not. More precisely, the resultant is zero if and only if the polynomials have a nontrivial common zero. We refer to \cite{MR2122859} for the details. Note that the determinant of an $n \times n$ matrix $A=[a_{ij}]$ is zero if and only if the homogeneous linear polynomials $l_0,l_1,...,l_n \in \mathbb{C}[x_0,...,x_n]$ have a nontrivial common zero, where each $l_j$ is the homogeneous linear polynomial with coefficients on the $j$-th column of $A$, i.e., $l_j=\sum_{i=1}^n a_{ij}x_i$. Thus, the following definition is natural.
\begin{definition}
    For an $n\times n \times n$ tensor $\mathcal{A} \in \mathbb{K}^{n \times n \times n}$, its \emph{resultant}, denoted by $\operatorname{Res}(\mathcal{A})$, is the resultant of the homogeneous polynomials $F_0,F_1,...,F_n$, where $F_i$ is the quadratic homogeneous polynomial corresponding to the (not necessarily symmetric) $i$-th lateral slice of $\mathcal{A}$ for all $i \in [n]$.
\end{definition}

With respect to this generalization, even for $M=I_n$ the resultantal (determinantal) identity does not hold in general because of the following example. We will use the resultant formula for two quadratic homogeneous polynomials $F_1=a_1x^2+a_2xy+a_3y^2$ and $F_2=b_1x^2+b_2xy+b_3y^2$:
\begin{equation*}
    \operatorname{Res}(F_1,F_2)= a_3^2b_1^2-a_2a_3b_1b_2+a_1a_3b_2^2+a_2^2b_1b_3-2a_1a_3b_1b_3-a_1a_2b_2b_3+a_1^2b_3^2.
\end{equation*}

\begin{example}\label{ex:CounterexampleResultant}
    Let $M=I_2$. We consider the following two $2\times2\times2$ tensors
    \begin{equation*}
        \mathcal{A}=\left[\begin{array}{cc|cc}
            2 & 1 & 4 & 1\\
            1 & 2 & 1 & 4
        \end{array}\right]~\text{and}~\mathcal{B}=\left[\begin{array}{cc|cc}
            2 & -1 & 2 & -1\\
            -1 & 2 & -1 & 2
        \end{array}\right].
    \end{equation*}
    Then
    \begin{equation*}
    \begin{aligned}
    \operatorname{Res}\bigl(\mathcal A\#_{*_M}\mathcal B\bigr)\approx 3.3287.
    \end{aligned}
    \end{equation*}
    On the other hand,
    \begin{equation*}
    \begin{aligned}
    \sqrt{\operatorname{Res}(\mathcal A)\,\operatorname{Res}(\mathcal B)}=0.
    \end{aligned}
    \end{equation*}
    Thus, 
    \begin{equation*}
    \operatorname{Res}\bigl(\mathcal A\#_{*_M}\mathcal B\bigr)
    \;\neq\;
    \sqrt{\operatorname{Res}(\mathcal A)\,\operatorname{Res}(\mathcal B)}.
    \end{equation*}
\end{example}

Another possible generalization of the determinant is Cayley's hyperdeterminant. In this paper, we only deal with Cayley's second hyperdeterminant.

\begin{definition}\label{def:hyper}
Let $\mathcal{A}=(a_{ijk})\in \mathbb{K}^{2\times2\times2}$. \emph{Cayley's second hyperdeterminant} of $\mathcal{A}$, denoted by $\operatorname{Det}(\mathcal{A})$, is defined as the following degree-4 polynomial
\begin{equation*}
\begin{aligned}
\operatorname{Det}(\mathcal{A})
&=a_{000}^2\,a_{111}^2
+a_{001}^2\,a_{110}^2
+a_{010}^2\,a_{101}^2
+a_{100}^2\,a_{011}^2\\
&-2\Bigl(
a_{000}a_{001}a_{110}a_{111}
+a_{000}a_{010}a_{101}a_{111}
+a_{000}a_{011}a_{101}a_{110}\\
&\quad\;+a_{001}a_{010}a_{101}a_{110}
+a_{001}a_{011}a_{100}a_{111}
+a_{010}a_{011}a_{100}a_{101}
\Bigr)\\
&+4
\bigl(
a_{000}a_{011}a_{101}a_{110}
+a_{001}a_{010}a_{100}a_{111}
\bigr).
\end{aligned}
\end{equation*}
\end{definition}

Again, even for $M=I_n$ the Cayley's hyperdeterminantal  identity does not hold in general for this generalization of the determinant. We will consider the tensors in Example \ref{ex:CounterexampleResultant}.

\begin{example}\label{ex:hyper}
    Let $M=I_2$, and take the tensors $\mathcal{A}$ and $\mathcal{B}$ in Example \ref{ex:CounterexampleResultant}.
    Then 
    \begin{equation*}
    \operatorname{Det}(\mathcal{A}\#_{*_M}\mathcal{B})\approx 0.8753,
    \end{equation*}
    and 
    \begin{equation*}
    \sqrt{\operatorname{Det}(\mathcal A)\operatorname{Det}(\mathcal B)}= 0.
    \end{equation*}
    Thus,
    \begin{equation*}
    \operatorname{Det}\bigl(\mathcal A\#_{*_M}\mathcal B\bigr)
    \;\neq\;
    \sqrt{\operatorname{Det}(\mathcal A)\,\operatorname{Det}(\mathcal B)}.
    \end{equation*}
\end{example}

\subsubsection{Wasserstein mean for pseudo-positive definite tensors}

Until now, we have proved that the tensor geometric mean under the $*_M$-product is the midpoint of the geodesic on the manifold of $*_M$-pseudo-positive-definite tensors for a specific Riemannian metric. Now, we will change the metric to another Riemannian metric and introduce the Wasserstein mean under the $*_M$-product of tensors. For the matrix case, let us change the Riemannian metric on $H^{n \times n}_{++}$ into $\hat{g}'$, locally defined by
\begin{equation*}
    \hat{g}'_P(A,B)=\sum_{i,j=1}^n s_i \frac{\operatorname{Re}(\overline{a_{ji}}b_{ji})}{(s_i+s_j)^2}
\end{equation*}
at $P$, where $P$ is diagonalized into $\operatorname{diag}(s_1,...,s_n)$ in an orthogonal basis, and $A,B$ are represented into $[a_{ij}],[b_{ij}]$ in the same basis. The motivation and construction of this Riemannian metric are introduced in \cite{MR3992484}. In that paper, it is proved that a geodesic from $A \in H^{n \times n}_{++}$ to $B \in H^{n \times n}_{++}$ is parametrized by
\begin{equation*}
    \gamma': t \in [0,1] \mapsto (1-t)^2A+t^2B+t(1-t)\left((AB)^{\frac{1}{2}}+(BA)^{\frac{1}{2}}\right)\in H^{n \times n}_{++}.
\end{equation*}
so that the Wasserstein mean is the midpoint of this geodesic.

Now, for the tensor case, we change the Riemannian metric on $\mathbb{H}^{n \times n \times p}_{++}$ into $g'$, locally defined by
\begin{equation*}
    g'_{\mathcal{P}}(\mathcal{A},\mathcal{B})=\sum_{i=1}^p \hat{g}'_{\widehat{\mathcal{P}}^{(i)}}\left((\widehat{A})^{(i)},(\widehat{B})^{(i)}\right)
\end{equation*}
at $\mathcal{P} \in \mathbb{H}^{n \times n \times p}_{++}$. Then, obviously, the geodesic $\gamma$ induces the geodesic
\begin{equation*}
    \Gamma':t \in [0,1] \mapsto (1-t)^2\mathcal{A}+t^2\mathcal{B}+t(1-t)\left((\mathcal{A}*_M\mathcal{B})^{\frac{1}{2}}+(\mathcal{B}*_M \mathcal{A})^{\frac{1}{2}}\right)  \in \mathbb{H}^{n \times n \times p}_{++}.
\end{equation*}
By taking its midpoint, we define the tensor Wasserstein mean.

\begin{definition}
For $\mathcal{A},\mathcal{B}\in\mathbb{H}^{n\times n\times p}_{++}$, The \textit{Wasserstein mean of $\mathcal{A}$ and $\mathcal{B}$ under the $*_M$-product}, denoted by $\mathcal{A}\diamond\mathcal{B}$, is defined as
\begin{equation}
\mathcal{A}\diamond\mathcal{B}=\frac{1}{4}\left(\mathcal{A}+\mathcal{B}+(\mathcal{A}*_M\mathcal{B})^{1/2}+(\mathcal{B}*_M\mathcal{A})^{1/2}\right).
\end{equation}
\end{definition}

In addition, for a fixed $t \in [0,1]$ it is also natural to define the weighted Wasserstein mean by
\begin{equation*}
    \mathcal{A}\diamond_t\mathcal{B}:=(1-t)^2\mathcal{A}+t^2\mathcal{B}+t(1-t)\left((\mathcal{A}*_M\mathcal{B})^{\frac{1}{2}}+(\mathcal{B}*_M \mathcal{A})^{\frac{1}{2}}\right).
\end{equation*}
Similarly to the geometric mean case, the multivariate Wasserstein mean for more than two $*_M$-pseudo-positive-definite tensors can be defined as the barycenter of those tensors in the Riemannian manifold.

In the case of tensor geometric means, Theorem \ref{thm:RiccatiTensorEquation} makes it possible to calculate it numerically and fast. However, to the best of our knowledge, no such approach is known for the tensor Wasserstein mean. 

\subsection{Pseudo-SVD under the $*_{M}$- product with injective maps and application on data compression} \label{Sec:applications_svd}

\subsubsection{Method}\label{Sec:SVD_method}
For the tensor-tensor multiplication under invertible and surjective maps, the authors of \cite{MR3394164} and \cite{keegan24} defined their corresponding $*_M$-SVD. However, since the tensor-tensor multiplication under injective maps is not associative (see Example \ref{ex:InjectiveNotAssociative}), it is not possible to define a proper $*_M$-SVD such that $\mathcal{A}=\mathcal{U} *_M \mathcal{S} *_M \mathcal{V}^H$.
We define, therefore, a tensor function that we call \emph{$*_M$-pseudo-SVD} that will also have similar applications in data science, which we will showcase with an experiment in data compression.

First, we introduce the concept of Frobenius norm of third-order tensors.

\begin{definition}\label{defn:frobeniusnorm}
Let $\mathcal{A}\in\mathbb{K}^{m\times n\times p}$. The \emph{Frobenius norm of third-order tensors} is defined slice-wise, as follows
\begin{equation*}
\| \mathcal{A}\|_F^2=\sum\limits_{k=1}^{p}\|\mathcal{A}_{:,:,k}\|_F^2.
\end{equation*}
    
\end{definition}

\begin{definition}\label{defn:pseudoSVD}
Let $M\in\mathbb{K}^{q\times p}$ be of full rank. We define a tensor function from $\mathbb{K}^{m\times n\times p}$ to $\mathbb{K}^{m\times n\times p}$, called the \emph{$*_M$-pseudo-$k$-SVD}, that sends a tensor $\mathcal{A}$ to 
\begin{equation*}
    \widetilde{\mathcal{A}}=\left[\begin{array}{c|c|c}
     \mathcal{\widehat{U}}^{(1)}_k\mathcal{\widehat{S}}^{(1)}_k(\mathcal{\widehat{V}}^{(1)}_k)^H & \cdots & \widehat{\mathcal{U}}^{(q)}_k\widehat{\mathcal{S}}^{(q)}_k(\widehat{\mathcal{V}}^{(q)}_k)^H
\end{array}\right]\times_3 M^+,
\end{equation*}
where $k\in [\operatorname{min}\{m,n\}]$ is the truncation parameter of the matrix $k$-SVD performed slice-wise to obtain the matrices $\widehat{\mathcal{U}}_k^{(i)}$, $\widehat{\mathcal{V}}_k^{(i)}$ and $\widehat{\mathcal{S}}_k^{(i)}$ for $i\in [q]$. Then,
$\widehat{\mathcal{U}}_k= \left[\begin{array}{c|c|c}
     \mathcal{\widehat{U}}^{(1)}_k & \cdots & \widehat{\mathcal{U}}^{(q)}_k
\end{array}\right]$ and $\widehat{\mathcal{V}}_k= \left[\begin{array}{c|c|c}
     \mathcal{\widehat{V}}^{(1)}_k& \cdots & \widehat{\mathcal{V}}^{(q)}_k
\end{array}\right]$ are stacks of unitary matrices, $\widehat{\mathcal{S}}_k= \left[\begin{array}{c|c|c}
     \mathcal{\widehat{S}}^{(1)}_k& \cdots & \widehat{\mathcal{S}}^{(q)}_k
\end{array}\right]$ is f-diagonal and the Frobenius norms of the diagonal tubes are ordered non-increasingly $ \|\widehat{\mathcal{S}}_{1,1,:}\|_F \geq \|\widehat{\mathcal{S}}_{2,2,:}\|_F \geq \cdots \geq \|\widehat{\mathcal{S}}_{l,l,:}\|_F \geq 0$, for $l = \min\{m, n\}$.
\end{definition}

We provide the pseudocodes in Algorithm \ref{algorithm:1} and Algorithm \ref{algorithm:2} for the algorithms of both the $*_M$-pseudo-SVD and the truncated $*_M$-pseudo-$k$-SVD.

\begin{algorithm}[h]
\caption{Full $*_M$-pseudo-SVD}
\label{algorithm:1}
\begin{algorithmic}[1]
\Require Tensor $\mathcal{A} \in \mathbb{K}^{m \times n\times p}$, injective matrix $M \in \mathbb{K}^{q \times p}$
\State Compute $\widehat{\mathcal{A}} = \mathcal{A} \times_3 M$
\For{$i = 1$ to $q$}
    \State Compute SVD of frontal slices $\widehat{\mathcal{A}}^{(i)} = U^{(i)} S^{(i)} (V^{(i)})^H$
\EndFor
\State Stack $U^{(i)}, S^{(i)}, V^{(i)}$ into tensors $\widehat{\mathcal{U}}, \widehat{\mathcal{S}}, \widehat{\mathcal{V}}$
\State Map back
\begin{equation*}
(\widehat{\mathcal{U}}\triangle\widehat{\mathcal{S}}\triangle\widehat{\mathcal{V}})\times_3 M^+
\end{equation*}
\end{algorithmic}
\end{algorithm}

\begin{algorithm}[h]
\caption{Truncated $*_M$-pseudo-SVD}
\label{algorithm:2}
\begin{algorithmic}[1]
\Require Tensor $\mathcal{A} \in \mathbb{K}^{m \times n\times p}$, matrix $M \in \mathbb{K}^{q \times p}$, $k \leq \min\{m, n\}$
\State Compute $\widehat{\mathcal{A}} = \mathcal{A} \times_3 M$
\State Compute SVD of frontal slices
    $\widehat{A}^{(i)}=U^{(i)} S^{(i)} (V^{(i)})^H$
\State $U^{(i)}, S^{(i)}, V^{(i)}$ into tensors $\widehat{\mathcal{U}}, \widehat{\mathcal{S}}, \widehat{\mathcal{V}}$
\State Truncate and perform face-wise multiplication
\begin{equation*}
\widehat{A}_k\approx\widehat{\mathcal{U}}(:,:,1:k)\triangle\widehat{\mathcal{S}}(1:k,1:k,:) \triangle \widehat{\mathcal{V}}^{H}(:,:,1:k)
\end{equation*}
\State Map back
\begin{equation*}
\widetilde{A} \approx \widehat{A}_k
\times_3 M^+
\end{equation*}
\end{algorithmic}
\end{algorithm}

\subsubsection{Experiment}\label{Sec:SVD_experiment}

We perform an experiment on data compression with injective matrix $M$ on the Indian Pines hyperspectral dataset \cite{PURR1947} and compare the results to those obtained in \cite{keegan24}. This dataset consists of agricultural imaging from northwestern Indiana and includes a $145 \times 145$ spatial grid with $220$ spectral bands. This means that instead of a traditional color image with RGB channels, this image has 220 channels from a 220-point spectrum. In our experiments we denote this cube by $\mathcal{A} \in \mathbb{R}^{m \times n \times p}=\mathbb{R}^{145 \times 145 \times 220}$,
and normalize it by its Frobenius norm
\begin{equation*}
\mathcal{A} \leftarrow \frac{\mathcal{A}}{\|\mathcal{A}\|_F},
\end{equation*}
so that subsequent error measurements correspond to relative spectral fidelity.

We will compare the relative compression rate at different truncation $k$, as well as the relative error, where these metrics are defined as 
\begin{equation*}
    RE=\frac{\| \mathcal{A}-\widehat{\mathcal{A}} \|_F}{\|\mathcal{A} \|_F}, \quad CR=\frac{\operatorname{st}(\mathcal{A})}{\operatorname{st}(\widehat{\mathcal{A}})+\operatorname{st}(M)},
\end{equation*}
with $\operatorname{st}(\cdot)$ referring to the computation of the storage cost of the object. Target values are relative errors close to zero and compression ratios bigger than one (that is, compression has been achieved).

To guarantee a fair comparison to the results in \cite{keegan24}, we have slightly modified their code. In their experiments with orthogonal surjective matrices, the matrix $M$ is actually initially defined as a $p\times p$ invertible orthogonal matrix. After doing the face-wise SVD of $\widehat{\mathcal{A}}=\mathcal{A}\times_3 M$, they stack the resultant matrices $U^{(i)},S^{(i)},V^{(i)}$ into tensors, truncate them and perform their face-wise multiplications, as follows 
\begin{equation*}
    U(: ,1:k, :  )\triangle S(1:k,1:k, : )\triangle V( : ,1:k, : ).
\end{equation*} 
Then, they slice the first $s$ columns of $M^+$ ($=M^\top$) before performing their mode-3 product. This process allows them to compare the results for different $p\times s$ matrices. 

\begin{remark}\label{rmk:no_slicing_pseudoinverse}
This slicing of the pseudoinverse is not possible for injective matrices, since $M^+(:,1:t)M(1:t,:)\neq I$, and thus we would not be implementing the $*_M$-product correctly. 

For the evaluation of the performance of the experiments with injective matrices, we do not slice the columns of the pseudoinverse, instead using its whole surjective pseudoinverse to perform the mode-3 product $\widehat A =A_k\times_3 M^+$. Then, for each $1\le s\le p$, we report the results of  
$\mathrm{err}(s)
=\bigl\lVert\,A(:,:,1:s)\;-\;\widehat A(:,:,1:s)\bigr\rVert_F$. 
\end{remark}

Comparing the results of our implementation directly to those reported in \cite{keegan24} would be impossible, since the two sets of results provide different information. As such, we implement their method with the full invertible orthogonal matrix and its inverse, and then measure how our approximation compares to this baseline. They performed experiments on several matrices, but we will compare to only $M=U_3$, where $U_3$ is the left singular matrix of the mode-3 unfolding of the tensor data $\mathcal{A}$. This choice of matrix was the one that yielded the best results in their paper due to it being a well-chosen data-dependent matrix.

By truncating both the original and reconstructed cubes to the same $n_1\times n_2\times ps$ block, we ensure a fair comparison of the surjective (square, invertible) and injective (rectangular) maps, without favoring one over the other. 

To minimize distorsion we have chosen an injective Jonhson-Lindenstrauss-type embedding.  A fast Johnson-Lindenstrauss transform was introduced for the first time in \cite{ailon2009fast}. Johnson-Lindenstrauss-type embeddings were then used in \cite{sarlos2006improved} to improve performance of previous SVD algorithms. All of these mappings are inspired by Johnson-Lindenstrauss' lemma \cite{johnson1984extensions}, which asserts that points in a high-order dimension can be projected into a much lower dimension while approximately preserving the distance between them. Since our map is injective, we are dealing with an oversampling task, but Johnson-Lindenstrauss-type embeddings can also be used in this case.

\begin{definition}
Let $\mathcal{A}\in\mathbb{R}^{m\times n\times p}$ be a tensor spectral image, with $p$ being the number of spectral bands, and fix an oversampling factor of~2.  We construct an injective Johnson–Lindenstrauss (JL)-type mapping $M\colon\mathbb R^{p}\to\mathbb R^{q}$, with $q>p$, similar to the one in \cite[Theorem 3]{sarlos2006improved}:

Set $N = 2^{\lceil \log_{2}(2p)\rceil}$ (we will take $N=512$) and $q = N$. Let $H_{q}\in\{\pm1\}^{q\times q}$ be the (Walsh–)Hadamard matrix, satisfying $H_{q}H_{q}^{T}=q\,I_{q}$. Form the diagonal sign‐matrix
\begin{equation*}
D \;=\;\operatorname{diag}(\xi_{1},\dots,\xi_{q})\;\in\;\mathbb R^{q\times q},
\end{equation*}
where $\xi_{1},\dots,\xi_{q}\in\{\pm1\}$.
Define the normalized signed‐Hadamard map as
\begin{equation*}
S=\frac{1}{\sqrt{p}}\,H_{q}\,D
    \;\in\;\mathbb R^{q\times q}.
\end{equation*}
Next, select uniformly at random a subset 
$R \subset [q]$, with $|R| = p$,
and let $P\in\{0,1\}^{p\times q}$ be the corresponding row‐selection matrix (each row of $P$ has exactly one 1, picking one index in $R$).
Finally, we define the injective JL‐type embedding as:
\begin{equation*}
M = (P\,S)^{\top}
    \in\mathbb R^{q\times p},
\end{equation*}
which has full column rank and, by standard Johnson–Lindenstrauss theory, preserves distances.
\end{definition}

Using an overcomplete injective matrix may prompt questions about how it compares to the identity map, since the injective map first embeds into a higher-dimensional space and then comes back to the same lower-dimensional space. To answer this, we will also compare the results to the performance of the identity matrix.

\subsubsection{Results}\label{Sec:results}

The results seen in Figure~\ref{fig:re_vs_cr}, Figure~\ref{fig:re_vs_cr_220}, Table~\ref{tab:inj_vs_surj_rel_err} and Table~\ref{tab:inj_vs_surj_cr} on the relative error and the compression ratio for both the injective case (JL) and the surjective case ($U_3$) suggest that while the surjective case made use of a data-dependent matrix and provided a more significant compression ratio (i.e., less storage required for the compressed image), the injective case is capable of providing slightly higher-fidelity reconstructions from $k=1$ to around $k=100$, even with a data-independent matrix, though with a lower compression ratio. For higher $k$ close to the maximum value $k=145$, both the surjective and injective models provide similar performances in terms of the compression ratio and the error.

\begin{table}[h]
  \centering
  \caption{Relative error for injective (JL), invertible ($U_3$) and identity ($I$) maps at various truncations \(k\) and slice counts \(s\).}
  \label{tab:inj_vs_surj_rel_err}
  \begin{tabular}{ccccc}
    \toprule
    \(k\) & \(s\)
      & \multicolumn{1}{c}{RE (JL)}
      & \multicolumn{1}{c}{RE (\(U_3\))}
      & \multicolumn{1}{c}{RE ($I$)}\\
    \midrule
    \multirow{5}{*}{1}
      & 1   & 0.0080     & 0.0081   & 0.0078     \\
      & 10  & 0.0241     & 0.0250   & 0.0241     \\
      & 50  & 0.0977     & 0.0996   & 0.0986   \\
      & 100 & 0.1124     & 0.1147   & 0.1132   \\
      & 220 & 0.1167     & 0.1188   & 0.1172    \\
    \addlinespace
    \hline
    \addlinespace
    \multirow{5}{*}{50}
      & 1   & 0.0035    & 0.0035   & 0.0037 \\
      & 10  & 0.0048    & 0.0052   & 0.0051  \\
      & 50  & 0.0119    & 0.0144   & 0.0146  \\
      & 100 & 0.0153    & 0.0180   & 0.0185 \\
      & 220 & 0.0167    & 0.0185   & 0.0186 \\
    \addlinespace
    \hline
    \addlinespace
    \multirow{5}{*}{100}
      & 1   & \num{6.08e-4}    & 0.0010   & 0.0011 \\
      & 10  & \num{9.56e-4}    & 0.0014   & 0.0015 \\
      & 50  & 0.0021    & 0.0034   & 0.0035  \\
      & 100 & 0.0028    & 0.0044   & 0.0045 \\
      & 220 & 0.0036    & 0.0045   & 0.0045 \\
    \addlinespace
    \hline
    \addlinespace
    \multirow{5}{*}{145}
      & 1   & \num{2.11e-16}  & \num{1.40e-16} &    \num{1.44e-16}\\
      & 10  & \num{8.09e-16}  & \num{7.62e-16} &     \num{1.28e-15}\\
      & 50  & \num{1.93e-15}  & \num{1.63e-15} &     \num{2.08e-15} \\
      & 100 & \num{2.54e-15}  & \num{2.08e-15} &     \num{2.84e-15} \\
      & 220 & \num{2.95e-15}  & \num{2.19e-15} &     \num{3.03e-15} \\
    \bottomrule
  \end{tabular}
\end{table}

\begin{table}[h]
  \centering
  \caption{Relative compression ratio for injective (JL), invertible ($U_3$) and identity ($I$) maps at various truncations $k$ and $s=220$.}
  \label{tab:inj_vs_surj_cr}
  \begin{tabular}{cccc}
    \toprule
    $k$
      & \multicolumn{1}{c}{CR (JL)}
      & \multicolumn{1}{c}{CR ($U_3$)} 
      & \multicolumn{1}{c}{CR ($I$)}\\
    \midrule
    \multirow{1}{*}{1}
      & 17.71     & 23.49    & 31.15\\
    \addlinespace
    \hline
    \addlinespace
    \multirow{1}{*}{20}
      & 1.50      & 1.53     & 1.56 \\
    \addlinespace
    \hline
    \addlinespace
    \multirow{1}{*}{30}
      & 1.01      & 1.03     & 1.04\\
    \addlinespace
    \hline
    \addlinespace
    \multirow{1}{*}{50}
      & 0.61      & 0.62     & 0.62\\
    \addlinespace
    \hline
    \addlinespace
    \multirow{1}{*}{145}
      & 0.21      & 0.21     & 0.22\\
    \bottomrule
  \end{tabular}
\end{table}

\begin{figure}[h]
  \centering
    \includegraphics[width=0.7\textwidth]{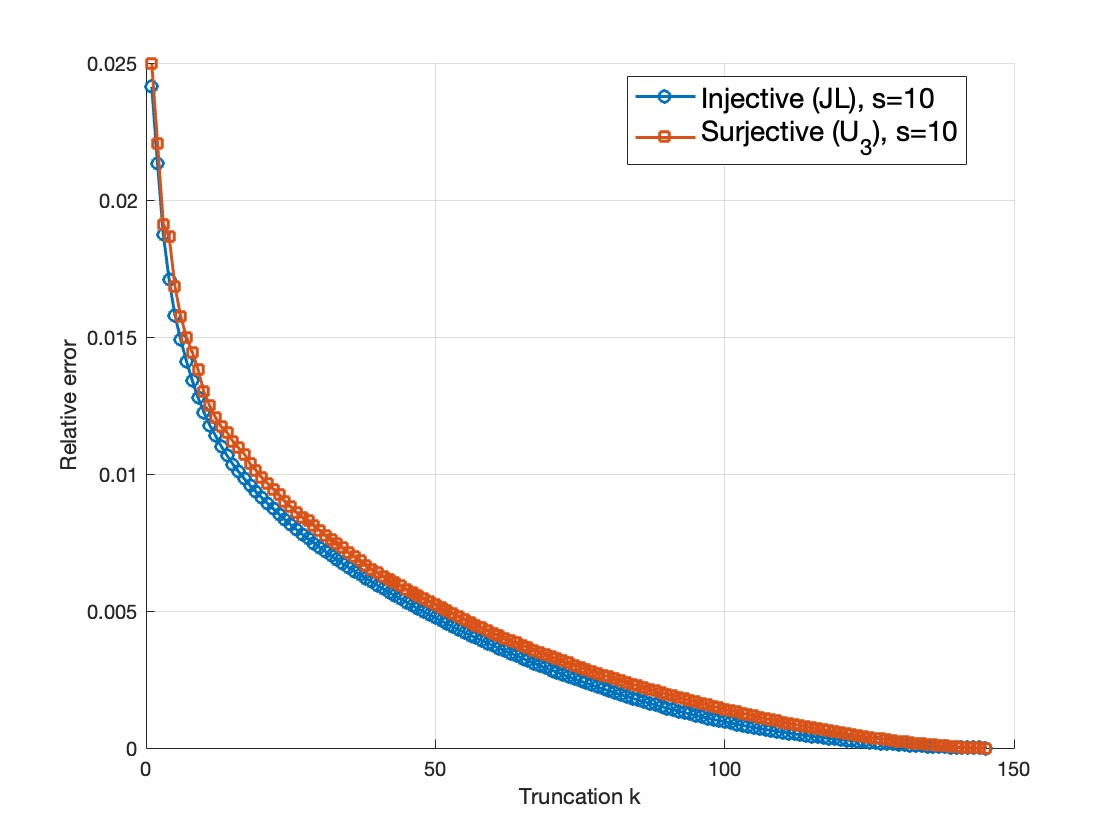}
    \caption{Comparison of the relative error with respect to the truncation $k$ for the injective (JL) and surjective/invertible ($U_3$) case with $s=10$.}
      \label{fig:re_vs_cr}
  \end{figure}
  
\begin{figure}

    \includegraphics[width=0.7\textwidth]{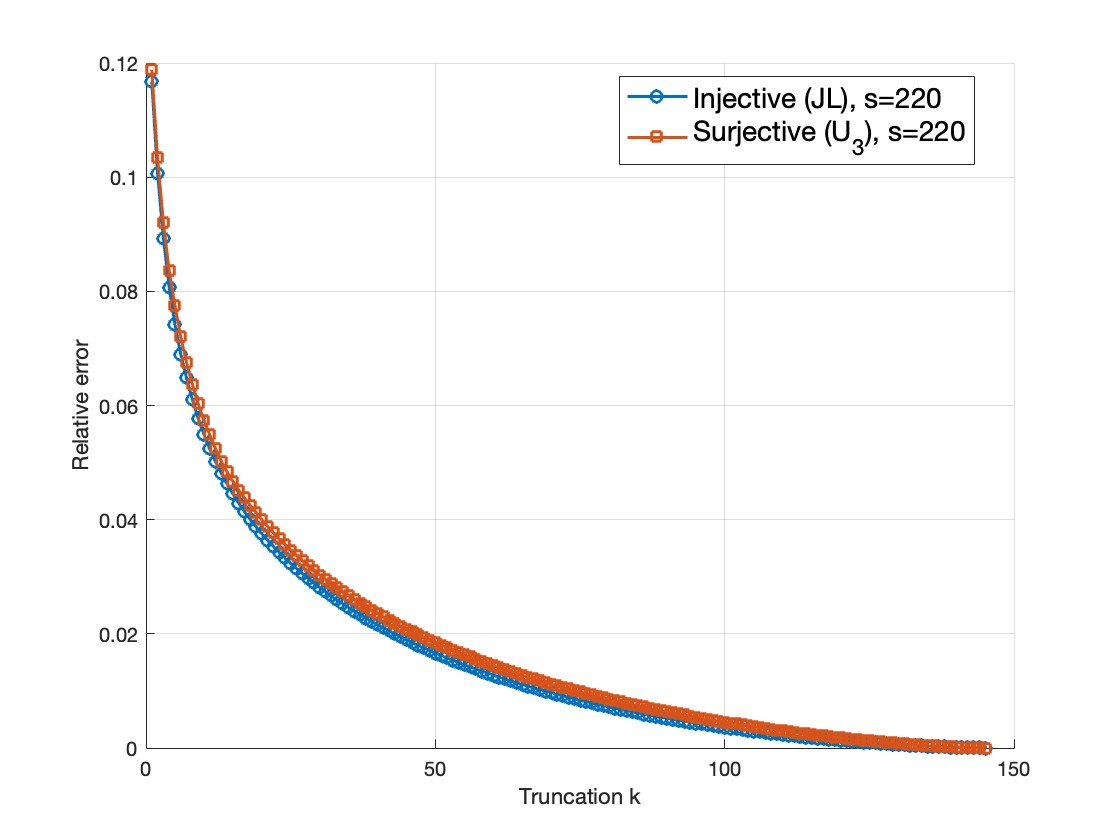}

    \caption{Comparison of the relative error with respect to the truncation $k$ for the injective (JL) and surjective/invertible ($U_3$) case with $s=220$.}
  \label{fig:re_vs_cr_220}
\end{figure}

\begin{figure}[h]
  \centering

  \begin{subfigure}[b]{\textwidth}
    \centering
    \begin{minipage}[b]{0.25\textwidth}
      \includegraphics[width=\textwidth]{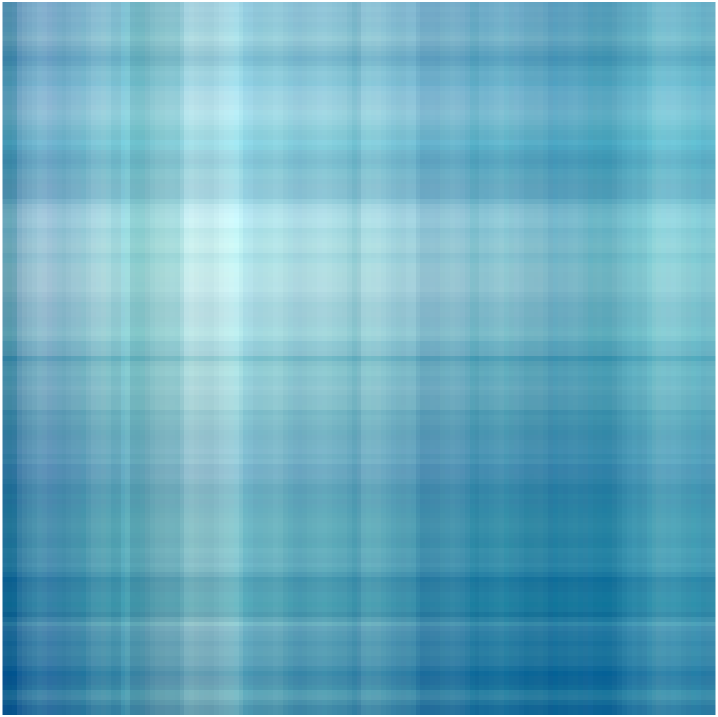}
    \end{minipage}\hspace{0.02\textwidth}
    \begin{minipage}[b]{0.25\textwidth}
      \includegraphics[width=\textwidth]{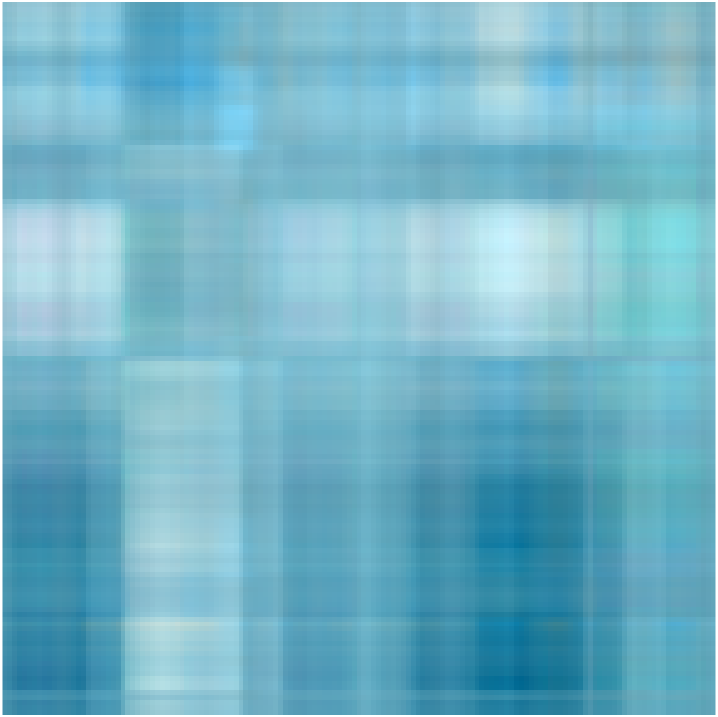}
    \end{minipage}
    \caption{$k=1, p=10$}
  \end{subfigure}

  \vspace{1em}

  \begin{subfigure}[b]{\textwidth}
    \centering
    \begin{minipage}[b]{0.25\textwidth}
      \includegraphics[width=\textwidth]{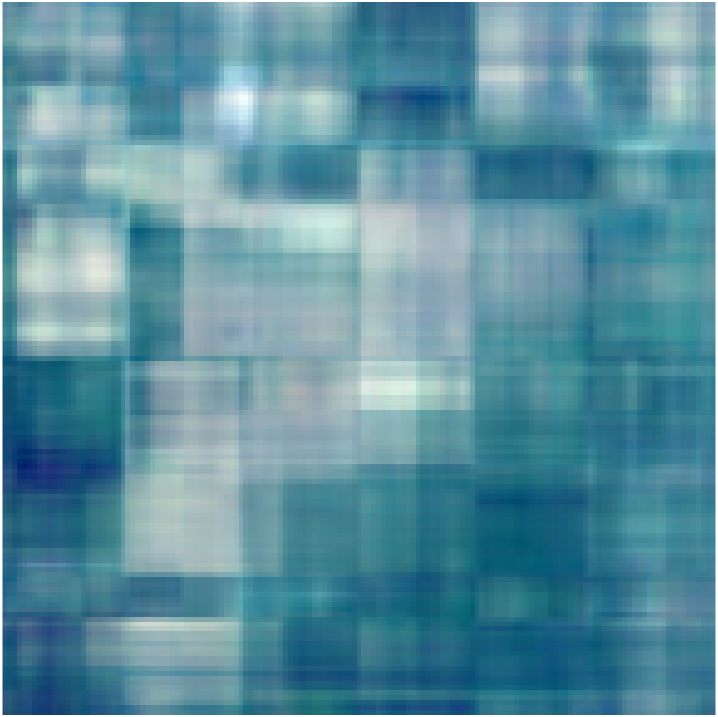}
    \end{minipage}\hspace{0.02\textwidth}
    \begin{minipage}[b]{0.25\textwidth}
      \includegraphics[width=\textwidth]{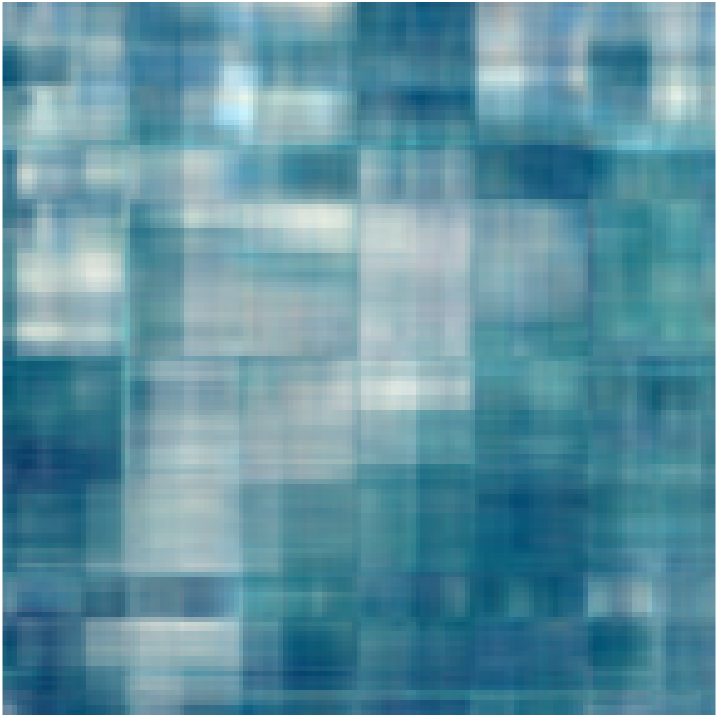}
    \end{minipage}
    \caption{$k=5, p=10$}
  \end{subfigure}

  \vspace{1em}

  \begin{subfigure}[b]{\textwidth}
    \centering
    \begin{minipage}[b]{0.25\textwidth}
      \includegraphics[width=\textwidth]{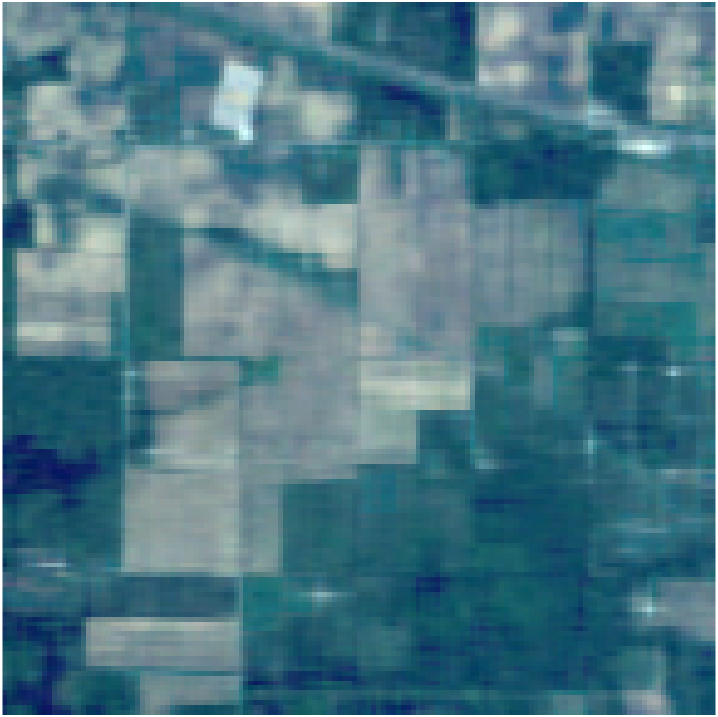}
    \end{minipage}\hspace{0.02\textwidth}
    \begin{minipage}[b]{0.25\textwidth}
      \includegraphics[width=\textwidth]{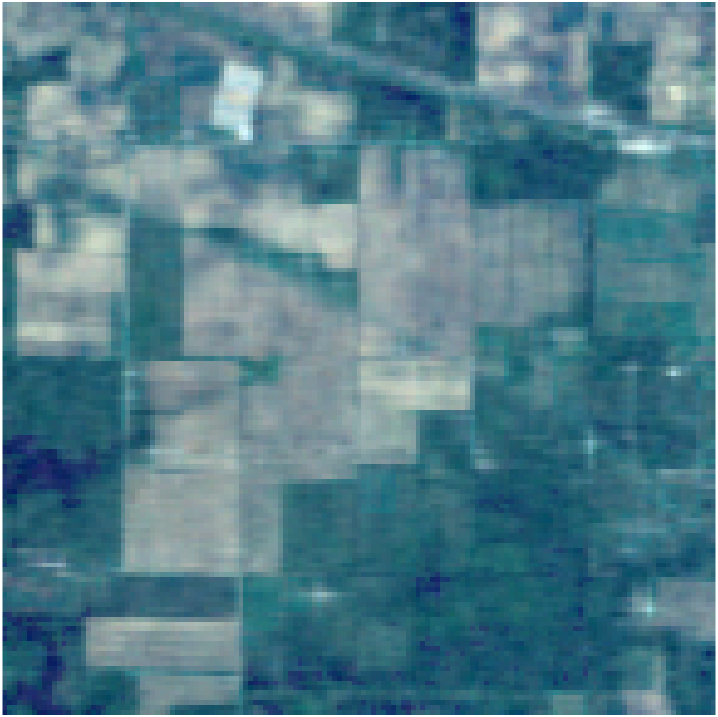}
    \end{minipage}
    \caption{$k=20, p=100$}
  \end{subfigure}

  \vspace{1em}

  \begin{subfigure}[b]{0.25\textwidth}
    \centering
    \includegraphics[width=\textwidth]{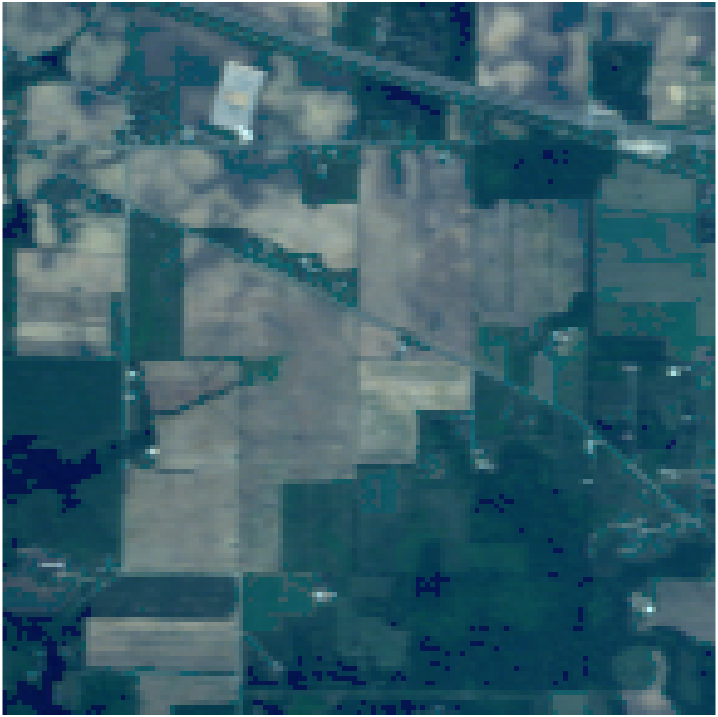}
    \caption{Original image}
  \end{subfigure}

  \caption{Truncation–slice compression snapshots. (Left) Using the injective JL‐type embedding, (right) using the surjective $U_3$ mapping, and (bottom) the original image.}
  \label{fig:injective_rgb_compare}
\end{figure}

It is worth noting that though we have showcased the performance for different $k$ and $s$, the slicing with $s$ is not done at the same step as in \cite{keegan24}, by Remark \ref{rmk:no_slicing_pseudoinverse}. In our case, both the original (normalized) image's channels and the (already calculated) approximation tensor's channels are sliced until $s$ and we keep track of how the error changes along the $s$ channels. Due to this construction, the error reported in this paper has a slightly different meaning as that in \cite{keegan24}.

The results on the identity matrix $I$ show that although the error starts similar to that of the injective case for small $k$, the injective case produces less error for higher $k$. This showcases that the injective matrix is actually doing something more than just sending the spectral bands to a higher-dimensional space and then pulling them back.

Figure \ref{fig:injective_rgb_compare} shows snapshots at different truncation stages $k$ for the JL injective matrix and the invertible $U_3$. All snapshots are taken at the channels in indices $(26,16,8)$, that is, the RGB values of the output image are taken from channel number 26, 16 and 8, respectively.

Thus, while the surjective case, by construction, provides advantages in terms of a bigger compression ratio, it is possible to conclude that the injective case is also capable of performing the $*_M$-pseudo-SVD successfully with relatively positive results. 

We provide the following reasoning for the difference in performance in both cases: the injective map performs an embedding into a higher-dimensional space (since $q>p$) and, thus, we have an overcomplete set of spectral bands. Literature suggests that overcomplete basis provide sparser representations \cite{overcomplete}. The surjective projection, on the other hand, and in particular the specific example in \cite{keegan24} with $M=U_3$, not only sends the slices to a lower-dimensional space, and thus, reducing the size of the stored data and resulting in a higher compression rate compared to that of the injective matrix, but the choice of $M$ also facilitates simplification when performing the face-wise SVD. 

We could argue, then, that each one should be used depending on the needs of the task: higher compression needed vs higher fidelity needed. 

\subsubsection{Conclusions}\label{Sec:conclusions} With this experiment we can conclude that, although the orthogonal case remains the most advantageous in terms of number of operations (storage costs) and compression rate, the injective case, in spite of its algebraic limitations - mainly, the lack of associativity -, also presents a tensor SVD-type operation.

\bibliography{ref.bib}

\begin{thebibliography}{10}

\bibitem{overcomplete}
Michal Aharon, Michael Elad, and Alfred Bruckstein.
\newblock K-svd: An algorithm for designing overcomplete dictionaries for
  sparse representation.
\newblock {\em Signal Processing, IEEE Transactions on}, 54:4311 -- 4322, 12
  2006.

\bibitem{ailon2009fast}
Nir Ailon and Bernard Chazelle.
\newblock The fast johnson--lindenstrauss transform and approximate nearest
  neighbors.
\newblock {\em SIAM Journal on computing}, 39(1):302--322, 2009.

\bibitem{MR4863478}
Josh Alman, Ran Duan, Virginia Vassilevska~Williams, Yinzhan Xu, Zixuan Xu, and
  Renfei Zhou.
\newblock More asymmetry yields faster matrix multiplication.
\newblock In {\em Proceedings of the 2025 {A}nnual {ACM}-{SIAM} {S}ymposium on
  {D}iscrete {A}lgorithms ({SODA})}, pages 2005--2039. SIAM, Philadelphia, PA,
  2025.

\bibitem{ando1978topics}
Tsuyoshi Ando.
\newblock Topics on operator inequalities.
\newblock {\em Hokkaido University, Sapporo}, 1978.

\bibitem{PURR1947}
Marion~F. Baumgardner, Larry~L. Biehl, and David~A. Landgrebe.
\newblock 220 band aviris hyperspectral image data set: June 12, 1992 indian
  pine test site 3, Sep 2015.

\bibitem{MR3443454}
Rajendra Bhatia.
\newblock {\em Positive definite matrices}.
\newblock Princeton Series in Applied Mathematics. Princeton University Press,
  Princeton, NJ, paperback edition, 2007.

\bibitem{MR2198952}
Rajendra Bhatia and John Holbrook.
\newblock Riemannian geometry and matrix geometric means.
\newblock {\em Linear Algebra Appl.}, 413(2-3):594--618, 2006.

\bibitem{MR3992484}
Rajendra Bhatia, Tanvi Jain, and Yongdo Lim.
\newblock On the {B}ures-{W}asserstein distance between positive definite
  matrices.
\newblock {\em Expo. Math.}, 37(2):165--191, 2019.

\bibitem{MR2944035}
Rajendra Bhatia and Rajeeva~L. Karandikar.
\newblock Monotonicity of the matrix geometric mean.
\newblock {\em Math. Ann.}, 353(4):1453--1467, 2012.

\bibitem{MR2896454}
Dario~A. Bini, Bruno Iannazzo, and Beatrice Meini.
\newblock {\em Numerical solution of algebraic {R}iccati equations}, volume~9
  of {\em Fundamentals of Algorithms}.
\newblock Society for Industrial and Applied Mathematics (SIAM), Philadelphia,
  PA, 2012.

\bibitem{blaeser13}
Markus Bl{\"a}ser.
\newblock {\em Fast Matrix Multiplication}.
\newblock Number~5 in Graduate Surveys. Theory of Computing Library, 2013.

\bibitem{MR2680253}
Karen Braman.
\newblock Third-order tensors as linear operators on a space of matrices.
\newblock {\em Linear Algebra Appl.}, 433(7):1241--1253, 2010.

\bibitem{MR1440179}
Peter B\"urgisser, Michael Clausen, and M.~Amin Shokrollahi.
\newblock {\em Algebraic complexity theory}, volume 315 of {\em Grundlehren der
  mathematischen Wissenschaften [Fundamental Principles of Mathematical
  Sciences]}.
\newblock Springer-Verlag, Berlin, 1997.
\newblock With the collaboration of Thomas Lickteig.

\bibitem{MR2122859}
David~A. Cox, John Little, and Donal O'Shea.
\newblock {\em Using algebraic geometry}, volume 185 of {\em Graduate Texts in
  Mathematics}.
\newblock Springer, New York, second edition, 2005.

\bibitem{MR274293}
John~E. Hopcroft and Leslie~Robert Kerr.
\newblock On minimizing the number of multiplications necessary for matrix
  multiplication.
\newblock {\em SIAM J. Appl. Math.}, 20:30--36, 1971.

\bibitem{johnson1984extensions}
William~B Johnson, Joram Lindenstrauss, et~al.
\newblock Extensions of lipschitz mappings into a hilbert space.
\newblock {\em Contemporary mathematics}, 26(189-206):1, 1984.

\bibitem{ju2024}
Jeong-Hoon Ju, Taehyeong Kim, Yeongrak Kim, and Hayoung Choi.
\newblock Geometric mean for t-positive definite tensors and associated
  riemannian geometry, 2024.

\bibitem{keegan24}
Katherine Keegan and Elizabeth Newman.
\newblock Projected tensor-tensor products for efficient computation of optimal
  multiway data representations, 2024.

\bibitem{MR3394164}
Eric Kernfeld, Misha Kilmer, and Shuchin Aeron.
\newblock Tensor-tensor products with invertible linear transforms.
\newblock {\em Linear Algebra Appl.}, 485:545--570, 2015.

\bibitem{MR4304063}
Misha~E. Kilmer, Lior Horesh, Haim Avron, and Elizabeth Newman.
\newblock Tensor-tensor algebra for optimal representation and compression of
  multiway data.
\newblock {\em Proc. Natl. Acad. Sci. USA}, 118(28):Paper No. e2015851118, 12,
  2021.

\bibitem{MR2794595}
Misha~E. Kilmer and Carla~D. Martin.
\newblock Factorization strategies for third-order tensors.
\newblock {\em Linear Algebra Appl.}, 435(3):641--658, 2011.

\bibitem{MR563399}
Fumio Kubo and Tsuyoshi Ando.
\newblock Means of positive linear operators.
\newblock {\em Math. Ann.}, 246(3):205--224, 1979/80.

\bibitem{MR2865915}
Joseph~M. Landsberg.
\newblock {\em Tensors: geometry and applications}, volume 128 of {\em Graduate
  Studies in Mathematics}.
\newblock American Mathematical Society, Providence, RI, 2012.

\bibitem{MR1864051}
Jimmie~D. Lawson and Yongdo Lim.
\newblock The geometric mean, matrices, metrics, and more.
\newblock {\em Amer. Math. Monthly}, 108(9):797--812, 2001.

\bibitem{MR4144013}
Lek-Heng Lim and Ke~Ye.
\newblock Ubiquity of the exponent of matrix multiplication.
\newblock In {\em I{SSAC}'20---{P}roceedings of the 45th {I}nternational
  {S}ymposium on {S}ymbolic and {A}lgebraic {C}omputation}, pages 8--11. ACM,
  New York, [2020] \copyright 2020.

\bibitem{MR2954480}
Yongdo Lim.
\newblock Factorizations and geometric means of positive definite matrices.
\newblock {\em Linear Algebra Appl.}, 437(9):2159--2172, 2012.

\bibitem{hyperspectral}
MATLAB.
\newblock {Hyperspectral Image Processing}.
\newblock
  \url{https://ww2.mathworks.cn/help/images/hyperspectral-image-processing.html}.

\bibitem{MR2137480}
Maher Moakher.
\newblock A differential geometric approach to the geometric mean of symmetric
  positive-definite matrices.
\newblock {\em SIAM J. Matrix Anal. Appl.}, 26(3):735--747, 2005.

\bibitem{newman2025projectedproducts}
Elizabeth Newman.
\newblock {projected-products}: Projected tensor–tensor products for
  efficient computation of optimal multiway data representations.
\newblock \url{https://github.com/elizabethnewman/projected-products}, 2025.
\newblock GitHub repository.

\bibitem{MR420302}
Wiesław Pusz and Stanisław~Lech Woronowicz.
\newblock Functional calculus for sesquilinear forms and the purification map.
\newblock {\em Rep. Mathematical Phys.}, 8(2):159--170, 1975.

\bibitem{sarlos2006improved}
Tamas Sarlos.
\newblock Improved approximation algorithms for large matrices via random
  projections.
\newblock In {\em 2006 47th annual IEEE symposium on foundations of computer
  science (FOCS'06)}, pages 143--152. IEEE, 2006.

\bibitem{MR248973}
Volker Strassen.
\newblock Gaussian elimination is not optimal.
\newblock {\em Numer. Math.}, 13:354--356, 1969.

\bibitem{matlab2024b}
{The MathWorks, Inc.}
\newblock {\em {MATLAB R2024b}}.
\newblock The MathWorks, Inc., Natick, MA, USA, 2024.
\newblock Version R2024b.

\bibitem{MR297115}
Shmuel Winograd.
\newblock On multiplication of {$2\times 2$} matrices.
\newblock {\em Linear Algebra Appl.}, 4:381--388, 1971.

\bibitem{MR3749414}
Ke~Ye and Lek-Heng Lim.
\newblock Fast structured matrix computations: tensor rank and {C}ohn-{U}mans
  method.
\newblock {\em Found. Comput. Math.}, 18(1):45--95, 2018.

\bibitem{MR4198732}
Meng-Meng Zheng, Zheng-Hai Huang, and Yong Wang.
\newblock T-positive semidefiniteness of third-order symmetric tensors and
  {T}-semidefinite programming.
\newblock {\em Comput. Optim. Appl.}, 78(1):239--272, 2021.

\end{thebibliography}

\end{document}